\documentclass[11pt,a4paper,twoside]{article}
\author{Stephen Bedford}
\title{Analysis of local minima for constrained minimization problems}

\usepackage{amsmath,amsfonts,euscript,amssymb,amsthm,times,graphicx,bm,bbm,fancyhdr,color}
\usepackage{txfonts,epsfig,subfig,mathrsfs,enumerate,textcomp,wrapfig,float,sidecap}
\usepackage[margin=1.0in]{geometry}
\usepackage[showonlyrefs]{mathtools}
\theoremstyle{plain}
\newtheorem{theorem}{Theorem}
\newtheorem{proposition}[theorem]{Proposition}
\newtheorem{definition}[theorem]{Definition}
\newtheorem{lemma}[theorem]{Lemma}
\newtheorem{corollary}[theorem]{Corollary}
\newtheorem{remark}[theorem]{Remark}

\newtheorem*{theorem*}{Theorem}
\newtheorem*{proposition*}{Proposition}
\newtheorem*{definition*}{Definition}
\newtheorem*{lemma*}{Lemma}
\newtheorem*{corollary*}{Corollary}
\newtheorem*{remark*}{Remark}

\newcommand{\vn}{{\bf n}}
\newcommand{\vv}{{\bf v}}
\newcommand{\vm}{{\bf m}}
\newcommand{\vw}{{\bf w}}

\newcommand{\vphi}{\boldsymbol\phi}
\newcommand{\vpsi}{\boldsymbol\psi}
\newcommand{\vQ}{{\bf Q}}

\renewenvironment{proof}{{\bf{Proof}\vspace{3 mm}}}{\qed}

\begin{document}

\maketitle

\begin{abstract}
We consider vectorial problems in the calculus of variations with an additional pointwise constraint. Our admissible mappings $\vn:\mathbb{R}^k\rightarrow \mathbb{R}^d$ satisfy $\vn(x)\in M$, where $M$ is a manifold embedded in Euclidean space. The main results of the paper all formulate necessary or sufficient conditions for a given mapping $\vn$ to be a weak or strong local minimizer. Our methods involve using projection mappings in order to build on existing, unconstrained, local minimizer results. We apply our results to a liquid crystal variational problem to quantify the stability of the unwound cholesteric state under frustrated boundary conditions.

\end{abstract}

%
%
%
% Introduction
%
%
%

\section{Introduction}\label{Introduction}

The canonical problem in the vectorial calculus of variations is to minimize the integral functional
\begin{equation}\label{0.1}
 I(\vn)=\int_\Omega F(x,\vn(x),\nabla\vn(x))\,dx
\end{equation}
where $\vn:\mathbb{R}^d\rightarrow \mathbb{R}^k$ lies in an appropriate Sobolev space together with some boundary conditions. Global minimizers for this problem exist provided that the Lagrangian $F$ satisfies certain quasiconvexity and growth conditions (see Dacorogna \cite{dacorogna2008direct} for a summary). For local minimizers much depends on which metric is used to measure locality. To formulate necessary and sufficient conditions for strong and weak local minimizers the first and second variations of the functional $I$ are considered, just as for finite-dimensional functions. Such local results were first proved by, for example, Graves \cite{graves1939weierstrass}, van Hove \cite{van1949signe}, Meyers \cite{meyers1965quasi} and Ball \& Marsden \cite{ball1986quasiconvexity}. A sufficiency result for strong local minimizers has proved to be the trickiest local result to derive, but the recent paper of Grabovsky \& Mengesha \cite{grabovsky2009sufficient} has shed new light on the subject. 

\vspace{3 mm}

However these classical results do not apply directly for variational problems with pointwise constraints. For example, in the study of liquid crystals, a continuum theory was proposed by Oseen \cite{oseen1933theory} in 1933 and expanded upon by Frank \cite{frank1958liquid} in later decades. They derived a continuum theory from a coarse-grained approach using the mean orientation $\vn(x)$ of the rod-like molecules as a macroscopic variable, in terms of which the energy of the system is defined. As a result, their problem was in the same form as \eqref{0.1} with the additional constraint that $\vn(x)\in \mathbb{S}^2$. The critical issue when working with constrained variational problems is what variations are admissible. For example, the standard unconstrained first variation is of the form
\begin{equation}\label{0.2}
 \left. \frac{d}{d\epsilon} I(\vn+\epsilon\vphi)\right|_{\epsilon=0}.
\end{equation}
However in the constrained case $\vn+\epsilon\vphi$ will not, in general, satisfy the constraint. This is why local minimizer results proved in the unconstrained case do not apply when there is a constraint. The focus of this paper is to address this issue by proving analogues of the unconstrained local minimizer results for a wide variety of constrained problems. These results are then applicable in the study of liquid crystals, micromagnetics and other constrained variational problems, providing the constraint is of sufficient regularity.

\vspace{3 mm}

In research on liquid crystals, designing multi-stable, power efficient, devices is one of the core aims in the industry, but the relevant stability questions are not usually considered within a precise mathematical framework. Hence local stability is typically approached through a linearization argument and phase plane analysis \cite{da2007mathematical}, or a study of the Euler-Lagrange equations \cite{davidson2002flexoelectric}, or simply through experimental data. This work supplements such analysis by providing a carefully formulated and rigorous mathematical structure in which to operate. In section \ref{sec:applications} we illustrate this by proving a new stability result for a widely studied liquid crystal problem. The one-constant Oseen-Frank energy for cholesterics is given by
\begin{equation}\label{0.2b}
 I(\vn)=K\int_\Omega |\nabla \vn|^2+2t\vn\cdot\nabla\times \vn+t^2\,dx,
\end{equation}
where $K,t>0$. The domain $\Omega$ is a cuboid, and we impose $\vn(x)={\bf e}_3$, on the top and bottom faces of $\Omega$. If the height of the cuboid is $d$, we show that $t=\frac{\pi}{d}$ is the critical value for the stability of the constant state; it is unstable if $t>\frac{\pi}{d}$ and stable if $t<\frac{\pi}{d}$. This prediction is consistent with the experimental data of Gartland et al.  \cite{gartland2010electric} once the relative sizes of their elastic constants $K_i$ are accounted for.

\vspace{3 mm}

More generally the minimization problem \eqref{0.1} takes on a greater level of complexity if the unknown $\vn:N\rightarrow M$ is a mapping between two manifolds $N\subset \mathbb{R}^d$ and $M\subset \mathbb{R}^k$. In this situation the corresponding Sobolev spaces are defined as
\begin{equation}\label{0.3}
 W^{1,p} \left( N,M\right) :=\left\{\left. \vn\in W^{1,p}\left( N,\mathbb{R}^k\right) \,\right|\, \vn(x)\in M\,\,\text{a.e.}\,\right\}.
\end{equation}
However, in this setting some of the standard notions about Sobolev spaces for mappings between Euclidean spaces do not hold. For example, Bethuel \& Zheng \cite{bethuel1988density} proved that depending on the topology of the manifold $M$, smooth maps are not always dense in $W^{1,p}\left( N,M \right)$. A standard example using these spaces is that of harmonic maps, where the object is to minimize
\begin{equation}\label{0.4}
  E(\vphi)=\int_N ||D\vphi||^2\,d\mu_N,
\end{equation}
where $d\mu_N$ is the measure on the manifold $N$ and we have the constraint $\vphi(x)\in M$ almost everywhere. The second variation of this functional was derived by Smith \cite{smith1975second} and has interesting properties. Xin \cite{xin1980some} showed that if $k>1$ then there is no non-constant stable harmonic map from $\mathbb{S}^{k-1}$ to any Riemannian manifold and there are many other related results \cite{leung1982stability,shen2004second}. Urakawa \cite{urakawa1993calculus} has a good overview of a number of such instability theorems. These results all emphasise that a constrained variational problem must be carefully considered in its own right, since its properties can be very different to its unconstrained analogue.

\vspace{3 mm}

In this paper we will explore some middle ground between the study of harmonic maps and the standard vectorial minimization problems. We restrict our domains to be bounded Lipschitz domains in $\mathbb{R}^d$, but we consider a much wider class of Lagrangians than in harmonic maps. The motivation for this work came from $\mathbb{S}^2$ valued maps, but we generalize our constraint to $\vn(x)\in M$ where $M\subset \mathbb{R}^k$ is a manifold of sufficient regularity. The regularity we impose on $M$ is necessary to ensure that the nearest point projection map $P:\mathbb{R}^k \rightarrow M$ has an appropriate differentiability. We need the regularity of $P$ because the first variation we will be using in this paper is of the form
\begin{equation}\label{0.5}
 \left. \frac{d}{d\epsilon} I(P(\vn+\epsilon\vphi))\right|_{\epsilon=0}.
\end{equation}
An advantage of this style of variation is that the test functions themselves are unconstrained. This makes the problem more accessible to numerical and computational methods for investigating stability.

\vspace{3 mm}

The paper contains a number of results, each of which either proves necessary or sufficient conditions for a given state to be a local minimizer. In section \ref{sec:weak necessary} we begin by establishing necessary conditions for a weak local minimizer. The first result, Theorem \ref{Weak Necessary}, is the usual result that a weak local minimizer must have a vanishing first variation and non-negative second variation. The second necessary condition for a weak local minimizer is a Legendre condition which is slightly modified from the standard version due to the curvature of the target manifold.

\vspace{3 mm}

From there we prove the weak sufficiency theorem, Theorem \ref{Weak sufficient}, that a vanishing first variation and strictly positive second variation yield a strict weak local minimizer. In section \ref{sec:strong necessary} we examine the extra necessary conditions which a strong local minimizer satisfies: quasiconvexity in the interior, at the boundary, and the Weierstrass condition. Then we establish two strong sufficiency results, Theorems \ref{Strong sufficient Taheri} and \ref{Strong sufficiency}, with different sets of assumptions. The first uses Taheri's result \cite[Theorem 3.3]{taheri2001sufficiency} with a pointwise Weierstrass condition, while the second uses Grabovsky \& Mengesha's result \cite{grabovsky2009sufficient}, which is much more technical, but also more general. We finish with two applications of our results. In the first we apply our theorems to the liquid crystal problem outlined above to establish the stability of the unwound cholesteric state as the twist parameter changes. The 
second illustrates how the study of global minimizers is very different with the addition of a constraint. We study the Dirichlet energy
\begin{equation}\label{0.6}
 I(\vn)=\int_0^1 |\vn_x|^2\,dx
\end{equation}
for $\vn:(0,1)\rightarrow \mathbb{S}^1$, and find infinitely many strong local minimizers which are not global minimizers. This cannot be the case without the constraint because in that situation the convexity of the Lagrangian ensures that the Euler-Lagrange equation has a unique solution.

%
%
%
% Introducing the problem
%
%
%

\section{Preliminaries and notation}\label{sec:prelims}

Throughout this paper, unless stated otherwise, $\Omega \subset \mathbb{R}^d$ denotes a bounded Lipschitz domain with boundary $\partial\Omega$. We suppose that the boundary has two disjoint, relatively open, measurable components $\partial\Omega_1$ and $\partial\Omega_2$ such that $\partial\Omega = \overline{\partial\Omega_1}\cup\overline{\partial\Omega_2}$. We will also be supposing throughout that $M\subset \mathbb{R}^k$ is a closed, bounded manifold of class $C^4$. The embedding theorems of Whitney and Nash ensure that these regularity assumptions can also apply to a wide array of abstract manifolds when embedded in Euclidean spaces. For $1\leqslant p \leqslant \infty$ we define the Sobolev space $W^{1,p}\left( \Omega,M\right)$ as 
\begin{equation}\label{1.0}
 W^{1,p}\left( \Omega,M\right):=\left\{\left.\, \vn\in W^{1,p}\left(\Omega,\mathbb{R}^k\right)\,\right|\, \vn(x)\in M\,\, \text{a.e.}\,\right\}.
\end{equation}
Then we consider the problem of minimizing 
\begin{equation}\label{1.1}
 I(\vn)=\int_\Omega F(x,\vn(x),\nabla\vn(x))\,dx,
\end{equation}
over the set of admissible mappings
\begin{equation}\label{1.2}
\mathcal{A}:=\left.\left\{\, \vn\in W^{1,1}\left(\Omega,M\right)\,\right|\,\vn=\vn_0\,\,\text{on}\,\,\partial\Omega_1\, \right\},
\end{equation}
where $k,d\in\mathbb{N}$ and $\vn_0 \in W^{1,1} \left( \Omega, M \right)$. We denote the corresponding set of test functions, or variations, by 
\begin{equation}\label{1.3}
\text{Var}_{\mathcal{A}}:=\left\{\,\left. \vv \in C^{\infty} \left({\Omega},\mathbb{R}^k\right)\,\right|\,\vv=0\,\,\text{on}\,\,\partial\Omega_1\,\right\}.
\end{equation}
We assume (unless stated otherwise) that 
\begin{itemize}
 \item $F \in C\left( \overline{\Omega}\times M\times\mathbb{R}^{k\times d}\right)$
 \item There exists some open neighbourhood $O$ of $M$ such that for every $x\in \Omega$,
\begin{equation}
F(x,\cdot,\cdot) \in C^2 \left(O\times \mathbb{R}^{k \times d}\right).
\end{equation}
\end{itemize}
These assumptions are sufficient to ensure that the second variation is well defined. In order to prove our results we will study unconstrained functionals which are related to $I$, so that we can apply standard results to them. However in order to do this we need a result about the regularity of the nearest point projection map. We remind the reader that a manifold $M\subset \mathbb{R}^k$ is a $C^r$-manifold, of codimension d, around some point $x\in M$, if there exists some open set $V\subset \mathbb{R}^k$ and $C^r$ function $F:V\rightarrow \mathbb{R}^d$ such that
\begin{equation}\label{1.4}
 M\cap V =\left\{ x \in V\,|\,F(x)=0\right\}.
\end{equation}

\begin{lemma}\cite[Lemma 4]{lewis2008alternating}\label{projectionregularity}
 Let ${M} \subset \mathbb{R}^k$ be a manifold of class $C^r$ (with $r\geqslant 2$) around some $x\in M$. Then there exists a $\delta>0$ such that the nearest point projection $P:B(x,\delta)\rightarrow {M}$ is unique and of class $C^{r-1}$.
\end{lemma}

It follows from the lemma, using a simple compactness argument, that for any closed, bounded manifold $M\subset \mathbb{R}^k$ of class $C^4$, there is some open set $U$, containing $M$, such that the nearest point projection $P:U\rightarrow M$ is unique and has $C^3$ regularity. We assume without loss of generality that $O\subset U$. We also choose $\delta^*>0$ such that if $d(x,M)<\delta^*$ then $x\in O$ and define a corresponding cut-off function $\psi \in C^\infty \left( \mathbb{R}^k,\mathbb{R} \right)$ such that 
\begin{equation}\label{1.5}
 \psi(x) \equiv 0 \quad \text{if} \,\, d(x,M)>\delta^* \quad \text{and} \quad \psi(x) \equiv 1 \quad \text{if} \,\, d(x,M)<\frac{\delta^*}{2}.
\end{equation}
Then we can define two associated functionals
\begin{equation}\label{1.6}
 J(\vm):=\int_\Omega G(x,\vm,\nabla \vm)\,dx=\int_\Omega \psi\left( \vm\right) F\left( x,P(\vm),\nabla \left[P(\vm)\right]\right) \, dx
\end{equation}
and
\begin{equation}\label{1.7}
\begin{split}
    K(\vm):= &\int_\Omega H(x,\vm,\nabla \vm)\,dx \\ =& J(\vm)+ \int_\Omega \psi(\vm)\left[\left| \vm -P(\vm) \right|^2 + \left| \nabla \vm-\nabla \left[P(\vm)\right] \right|^2 \right] \, dx.
\end{split}
\end{equation}
These functionals both act on the set 
\begin{equation}\label{1.8}
 \mathcal{B} := \left.\left\{\, \vm\in W^{1,1}\left(\Omega,\mathbb{R}^k\right)\,\right|\,\vm=\vn_0\,\,\text{on}\,\,\partial\Omega_1\, \right\}.
\end{equation}
Then from the definitions \eqref{1.5}, \eqref{1.6} and \eqref{1.7}, together with the assumptions on the Lagrangian $F$, we have that 
\begin{itemize}
 \item $G,H \in C \left(\overline{\Omega} \times \mathbb{R}^k \times \mathbb{R}^{k\times d} \right)$
 \item For every $x\in \Omega$, $G(x,\cdot,\cdot),H(x,\cdot,\cdot)\in C^2\left( \mathbb{R}^k\times \mathbb{R}^{k\times d}\right)$.
\end{itemize}
It is also worth noting that the set of variations for the functionals $J$ and $K$ are the same as for the original functional $I$ so that 
\begin{equation}\label{1.9}
 \text{Var}_{\mathcal{B}}=\text{Var}_{\mathcal{A}}.
\end{equation}

We will denote the unit ball in $\mathbb{R}^n$ simply by ${\bf B}$. The function space $C_0^\infty\left(U,\mathbb{R}^k\right)$ will denote the space of smooth maps from $U$ to $\mathbb{R}^k$ with compact support in $U$. For ease of notation if $\vv \in \text{Var}_\mathcal{A}$ we will often use the functions $\vw_t(x)$ and $g(t)$ defined as
\begin{equation}\label{1.10}
 \vw_t(x):=P(\vn(x)+t\vv(x))\quad \text{and} \quad g(t):=I(\vw_t).
\end{equation}
So $\vw_t$ represents our small perturbation and $g(t)$ is the energy of the perturbed state. We will use the notation that primes denote derivatives with respect to t:
\begin{equation}\label{1.12}
 \vw'_t(x):=\frac{d}{dt}\vw_t(x).
\end{equation}
We also need to define some theoretical notions which we will be studying in this paper.
\begin{definition}\label{def3}
 A function $\vn \in \mathcal{A}$ is a strong local minimizer of the functional $I$ if there exists an $\epsilon>0$ such that if $\vm\in\mathcal{A}$ and $||\vn-\vm||_{\infty}<\epsilon$ then $I(\vm)\geqslant I(\vn)$.
\end{definition}

\begin{definition}\label{def4}
 A function $\vn \in \mathcal{A}$ is a weak local minimizer of the functional $I$ if there exists an $\epsilon>0$ such that if $\vm\in\mathcal{A}$ and $||\vn-\vm||_{1,\infty}<\epsilon$ then $I(\vm)\geqslant I(\vn)$.
\end{definition}

\begin{definition}\label{def7}
 A function $\vn \in \mathcal{A}$ is a strict weak (respectively strong) local minimizer of the functional $I$ if there exists an $\epsilon>0$ such that if $\vm\in\mathcal{A}$ and $0<||\vn-\vm||_{1,\infty}<\epsilon$ (respectively $0<||\vn-\vm||_{\infty}<\epsilon$), then $I(\vm)> I(\vn)$.
\end{definition}

\begin{definition}\label{def1}
The Weierstrass excess function of a Lagrangian $F$ is given by
\begin{equation}\label{1.13}
 \mathcal{E}_F (x,z,p,q):=F(x,z,q)-F(x,z,p)-F_p(x,z,p):(q-p).
\end{equation}
\end{definition}

\begin{definition}[Quasiconvexity]\label{def2}
A continuous function $f:\mathbb{R}^{k\times d} \rightarrow \mathbb{R}$ is quasiconvex at a point $\zeta\in \mathbb{R}^{k\times d}$ if 
\begin{equation}\label{1.14}
  f(\zeta) \leqslant \frac{1}{|D|} \int_D f\left( \zeta +\nabla \vphi(x)\right) \,dx
\end{equation}
 for every set $D\subset \mathbb{R}^d$ which is open and bounded and every $\vphi \in W^{1,\infty}_0\left(D,\mathbb{R}^k \right)$.
\end{definition}

\begin{definition}[Rank-One Convexity]\label{def5}
 A continuous function $f:\mathbb{R}^{k\times d} \rightarrow \mathbb{R}$ is rank-one convex at a point $\zeta\in \mathbb{R}^{k\times d}$ if
 \begin{equation}\label{1.15}
  f(\zeta)\leqslant tf(A_1)+(1-t)f(A_2)
 \end{equation}
 for any $A_1,A_2 \in \mathbb{R}^{k\times d}$, $t\in [0,1]$, such that $\zeta= tA_1+(1-t)A_2$ and $\text{rank}\left(A_1-A_2\right)\leqslant 1$.
\end{definition}

In most of the results and proofs given below we will apply a result from the unconstrained theory to either the functional $J$ or $K$, and then relate the condition back to $I$. Many of the standard results can be found in Giaquinta and Hildebrandt \cite[Ch 4]{giaquinta1996calculus}, although where appropriate we will use the simple extension of their results from $C^1$ minimizers to $W^{1,\infty}$ minimizers. This is straightforward when studying the first and second variations, but we have to be more careful when proving point-wise conditions like quasiconvexity or the Legendre-Hadamard condition for $W^{1,\infty}$ functions. Indeed in section \ref{sec:strong necessary} we prove directly the unconstrained result that a Lipschitz function which is a strong local minimizer is quasiconvex in the interior. There is no explicit result concerning this in the literature that we are aware of.

\subsection{The Case $M=\mathbb{S}^{k-1}$}

In the situation where $M=\mathbb{S}^{k-1}$ it is instructive to note the explicit forms for these abstract concepts outlined in the preliminaries. The projection map $P:\mathbb{R}^k\setminus \left\{ 0\right\} \rightarrow \mathbb{S}^{k-1}$ is given by $P(x)=\frac{x}{|x|}$, so that if $x\in \mathbb{R}^k$ and $\vm \in \mathcal{B}$ then
\begin{equation}\label{1.16}
 \nabla P(x)=\frac{1}{|x|}I -\frac{x\otimes x}{|x|^3}\quad \text{and}\quad \nabla \left[ P(\vm)\right] = \frac{\nabla\vm}{|\vm|}-\frac{\vm\otimes\left( \vm^T \nabla \vm \right)}{|\vm|^3}.
\end{equation}
In this instance we also have explicit forms for the perturbations. Using the definitions in \eqref{1.10}, if $\vn \in \mathcal{A}$ and $\vv\in \text{Var}_{\mathcal{A}}$, then
\begin{equation}\label{1.11}
 \vw_0=\vn,\quad \vw'_0=\vv-\vn\left(\vn\cdot \vv\right), \quad \vw''_0=-2\vv\left( \vn\cdot \vv\right) -\vn |\vv|^2+3\vn\left(\vn\cdot\vv\right)^2.
\end{equation}
All of the results in sections \ref{sec:weak necessary}-\ref{sec:strong sufficient} can be understood for mappings into spheres by using these relations in the statements rather than the more general forms for mappings into differentiable manifolds.

\begin{remark}\label{Curvature remark}
If $x\in M$ then $\nabla P(x)$ is the matrix representation of the orthogonal projection onto the tangent space of $M$ at $x$. Similarly $\vw'_0(x)$ will always lie in the tangent space to $M$ at $\vn(x)$. This can help with the interpretation of a number of the statements below. For example, in Theorem \ref{Weak sufficient} we see that the positivity of the second variation we require is in fact positivity in all tangential directions of $M$ at $\vn(x)$.
\end{remark}

%
%
%
%
% Necessary conditions for a weak local minimizer 
%
%
%
%

\section{Necessary conditions for a weak local minimizer}\label{sec:weak necessary}

\begin{theorem}\label{Weak Necessary}
 Suppose that $\vn\in \mathcal{A}\cap W^{1,\infty}\left(\Omega,M\right)$ is a weak local minimizer of $I$ and for $v\in \text{Var}_\mathcal{A}$ define (see \eqref{1.10})
 \begin{equation}\label{3.1}
  g(t)=I\left( P(\vn+t\vv)\right).
 \end{equation}
 Then 
 \begin{equation}\label{3.2}
  g'(0)=0 \quad \text{and} \quad g''(0)\geqslant 0 \quad \forall \vv \in \text{Var}_{\mathcal{A}}.
 \end{equation}
 \end{theorem}
 
 \begin{proof}
 
 Since $\vn$ is a weak local minimizer, we know that there exists some $\alpha>0$ such that
 \begin{equation}\label{3.3}
  I(\vm)\geqslant I(\vn) \,\,\,\,\text{for all}\,\,\,\, \vm \in \mathcal{A} \,\,\,\, \text{such that} \,\,\,\, ||\vn-\vm||_{1,\infty} < \alpha.
 \end{equation}
 Take $\vm \in \mathcal{B}$, and without loss of generality we assume that $||\vn-\vm||_\infty\leqslant \frac{\delta^*}{2}$. We first obtain an upper bound for $\left|\left| \vn-P(\vm) \right|\right|_{1,\infty}$ in terms of $\left|\left| \vn-\vm\right|\right|_{1,\infty}$. We do this with two straightforward calculations. 
 \begin{equation}\label{3.4}
 \begin{split}
  \left| \vn-P(\vm) \right| &\leqslant |\vn-\vm|+|\vm-P(\vm)| \\
  &\leqslant 2|\vn-\vm|
 \end{split}
 \end{equation}
 since $P(\vm)$ is the closest point projection onto $M$. As for the derivatives we first note the relation
 \begin{equation}\label{3.5}
  \nabla \left[P(\vm(x))\right]=\nabla P (\vm(x))\,\nabla\vm(x),
 \end{equation}
 with the right-hand side of \eqref{3.5} being understood as two matrices multiplied together. As $P\in C^3$ and $||\vn-\vm||_\infty<\infty$, we use \eqref{3.5} to deduce $\nabla \left[ P(\vm)\right] \in L^1$ so that $P(\vm)\in \mathcal{A}$. Furthermore
 \begin{equation}\label{3.6}
 \begin{split}
   |\nabla\vn-\nabla \left[P(\vm)\right]|&
   =\left| \nabla P(\vm)\nabla \vm-\nabla P(\vm) \nabla\vn+\nabla P(\vm)\nabla \vn-\nabla \vn\right|\\
   &\leqslant |\nabla P(\vm)|\,|\nabla\vm-\nabla\vn|+|\nabla\vn|\,|\nabla P(\vm)-\nabla P(\vn)|.
 \end{split}
 \end{equation}

 In order to bound these terms we use the fact that the projection map $P$ is twice differentiable and the mean value theorem to see that 
 \begin{equation}\label{3.7}
   |\nabla \vn-\nabla\left[P(\vm)\right]| 
   \leqslant C_1\left( |\nabla\vm-\nabla\vn|+|\vm-\vn|\right).
 \end{equation}
 Combining \eqref{3.4} and \eqref{3.7} tells us that $\left|\left| \vn-P(\vm) \right|\right|_{1,\infty} \leqslant D\left|\left| \vn-\vm \right|\right|_{1,\infty}$, where $D$ is independent of $\vm$. Therefore if $\left|\left| \vn-\vm \right|\right|_{1,\infty}$ is sufficiently small 
 \begin{equation}\label{3.8}
   J(\vm)=I\left( P(\vm)\right) \geqslant I(\vn)=J(\vn).
 \end{equation}
 In other words $\vn$ is a weak local minimizer of $J$. We note here that \eqref{3.8} holds because $||\vm-\vn||_{\infty} \leqslant\frac{\delta^*}{2}$ implies $\psi(\vm)\equiv1$. Now we can apply the standard result \cite[Ch 4,1.1]{giaquinta1996calculus} to say that 
 \begin{equation}\label{3.9}
  \delta J(\vn)(\vv) :=\left.\frac{d}{dt}J(\vn+t\vv)\right|_{t=0}=0
 \end{equation}
 and 
 \begin{equation}\label{3.10}
  \delta^2 J(\vn)(\vv,\vv):=\left.\frac{d^2}{dt^2}J(\vn+t\vv)\right|_{t=0}\geqslant 0,
 \end{equation}
 for all $\vv\in \text{Var}_{\mathcal{A}}$. In order to relate these conditions to $I$, we first calculate that for all $\vv\in \text{Var}_{\mathcal{A}}$
 \begin{equation}\label{3.11}
  \left.\frac{d}{dt}\psi(\vn+t\vv)\right|_{t=0}= \left.\frac{d^2}{dt^2}\psi(\vn+t\vv)\right|_{t=0}=0,
 \end{equation}
 so that the construction of the functional $J$ implies, as required, that
 \begin{equation}\label{3.12}
  g'(0)=\left.\frac{d}{dt} I\left( P(\vn+t\vv) \right)\right|_{t=0}= \left.\frac{d}{dt}J(\vn+t\vv)\right|_{t=0}=0
 \end{equation}
 and 
 \begin{equation}\label{3.13}
  g''(0)=\left.\frac{d^2}{dt^2} I\left( P(\vn+t\vv) \right)\right|_{t=0}= \left.\frac{d^2}{dt^2}J(\vn+t\vv)\right|_{t=0}\geqslant 0. 
 \end{equation}
 \qed
\end{proof}

\begin{theorem}[Legendre's Condition]\label{Legendre}
 Suppose that $\vn\in \mathcal{A}\cap W^{1,\infty}\left(\Omega,M\right)$ is a weak local minimizer of $I$ and denote $F(x,\vn(x),\nabla\vn(x))$ by $F$, then
 \begin{equation}\label{3.14}
  \sum_{i,j,k,l}\frac{\partial^2F}{\partial \vn_{i,j}\vn_{k,l}}\eta_j\eta_l \left[ \nabla P(\vn)\zeta\right]_i\left[ \nabla P(\vn)\zeta\right]_k \geqslant 0,
 \end{equation}
 for all $\eta\in \mathbb{R}^d$, $\zeta\in \mathbb{R}^{k}$ and a.e. $x\in\Omega$.

\end{theorem}

\begin{proof}
 
We take an arbitrary $\vphi \in C^{\infty}_0 \left( \Omega ,\mathbb{R}^{k}\right)$ and let $\vw_t=\vn+t\vphi$. As part of this proof we need an explicit form for the second variation. Therefore with a brief calculation we find
 \begin{equation}\label{3.15}
 \begin{split}
  L(\vphi):&=\int_\Omega W(x,\vphi,\nabla\vphi)\,dx \\
  &=\left. \frac{d^2}{dt^2} I\left( \vw_t\right)\right|_{t=0} \\
  &=\int_{\Omega} {\vw'_0}^T F_{\vn\vn}(x)\vw'_0 +2\left[ F_{\vn\nabla\vn}(x) \nabla \vw'_0\right]\cdot \vw'_0 + \nabla {\vw'_0}^T F_{\nabla\vn \nabla\vn}(x) \nabla\vw'_0 \\
  & \quad + F_{\vn}(x)\cdot \vw_0''+F_{\nabla\vn}(x):\nabla \vw''_0\,dx
  \end{split}
 \end{equation}
 where
 \begin{equation}\label{3.16}
 \begin{array}{c}
  F(x):=F(x,\vn(x),\nabla\vn(x)),  \\ \\
  \vw'_0(x)_i=\sum_j P_{i,j}(\vn(x))\vphi_j(x),\\ \\
  \vw_0{''}(x)_i=\sum_{j,k} P_{i,jk}(\vn(x))\vphi_{j}(x)\vphi_{k}(x). 
 \end{array}
 \end{equation}
 The summation notations we are using in \eqref{3.15} are
  \begin{equation}\label{3.17} 
  \left(F_\vn\right)\cdot {\bf b}=\frac{\partial F}{\partial \vn_i}{\bf b}_i, \quad 
  \left(F_{\nabla\vn}\right):{\bf A}=\frac{\partial F}{\partial \vn_{i,j}}{\bf A}_{i,j}, \quad {\bf b}^T\left(F_{\vn\vn}\right){\bf b}=\frac{\partial^2 F}{\partial \vn_i \partial \vn_j}{\bf b}_i{\bf b}_j, 
 \end{equation}
 and 
 \begin{equation}\label{3.18}
  \left[\left(F_{\vn\nabla\vn}\right){\bf A}\right]\cdot{\bf b}=\frac{\partial^2 F}{\partial \vn_i \partial \vn_{j,k}}{\bf A}_{j,k}{\bf b}_i, \quad
  {\bf A}^T\left(F_{\nabla\vn\nabla\vn}\right){\bf A}=\frac{\partial^2 F}{\partial \vn_{i,j} \partial \vn_{k,l}}{\bf A}_{i,j}{\bf A}_{k,l}.
 \end{equation}
 
 We know from Theorem \ref{Weak Necessary} that since $\vn\in \mathcal{A}\cap W^{1,\infty}\left(\Omega,M\right)$ is a weak local minimizer of $I$, its second variation is non-negative. Therefore
 \begin{equation}\label{3.19}
L(\vphi)\geqslant 0 \quad \forall \vphi\in C_0^\infty \left( \Omega, \mathbb{R}^{k}\right).
 \end{equation}
 By using a simple density argument we can expand upon \eqref{3.19} and say
 \begin{equation}\label{3.20}
  L(\vphi)\geqslant 0 \,\,\,\, \forall \vphi \in W^{1,2}_0\left( \Omega,\mathbb{R}^k\right).
 \end{equation}
 So certainly, $\vphi=0$ is a strong local minimizer of the functional $L$. Now we can apply Corollary \ref{Weierstrass} (which will be proved in section \ref{sec:strong necessary}) to deduce that for a.e. $x\in \Omega$
 \begin{equation}\label{3.21}
  \mathcal{E}_W\left(x,0,0,\zeta\otimes \eta\right)\geqslant 0 \,\,\,\, \forall \, \zeta \in \mathbb{R}^k,\,\,\eta\in \mathbb{R}^d. 
 \end{equation}
 Since $W$ is a quadratic function in $\vphi$ we know that $W_{\nabla\vphi}(x,0,0)=0$. Therefore we simplify \eqref{3.21} to find
 \begin{equation}\label{3.22}
  \begin{split}
   0&\leqslant \mathcal{E}_W\left(x,0,0,\zeta\otimes \eta\right)\\
   &=W(x,0,\zeta\otimes \eta)\\
   &=\sum_{i,j,k,l} \frac{\partial^2 F}{\partial \vn_{i,j}\partial \vn_{k,l}} \eta_j\eta_l \left[ \nabla P(\vn)\zeta\right]_i\left[ \nabla P(\vn)\zeta\right]_k.
  \end{split}
 \end{equation}
\qed
\end{proof}

%
%
%
%
% Sufficient conditions for a weak local minimizer 
%
%
%
%

\section{Sufficient conditions for a weak local minimizer}\label{sec:weak sufficient}

\begin{theorem}\label{Weak sufficient}
 
 Let $\vn \in\mathcal{A}\cap W^{1,\infty}\left(\Omega,M\right)$ and for $v\in \text{Var}_\mathcal{A}$ define (see \eqref{1.10})
 \begin{equation}\label{4.1}
  g(t)=I\left(P(\vn+t\vv)\right)=I(\vw_t).
 \end{equation}
 Suppose there exists some $\gamma>0$ such that
 \begin{equation}\label{4.2}
  g'(0)=0\quad\text{and}\quad g''(0)\geqslant \gamma ||\vw'_0||_{1,2}^2\quad \forall \vv \in \text{Var}_{\mathcal{A}}.
 \end{equation}
 Then $\vn$ is a strict weak local minimizer of $I$.
 \end{theorem}
 
\begin{proof}

For the sufficiency results we will use the functional $K$ as the unconstrained counterpart to $I$. The reason that we cannot use $J$ as we did in the necessity proofs, is that the inherent structure of $J$ means it cannot have a strict local minimizer. This is because if we take any two functions $\vm_1, \vm_2 \in \mathcal{B}$ such that $P(\vm_1)=P(\vm_2)$, then $J(\vm_1)=J(\vm_2)$. We will use our assumptions to show that at $\vn$, the first variation of $K$ is zero and its second variation is strictly positive. We take a test function $\vv \in \text{Var}_{\mathcal{A}}$ and recall that the function $\vw_t$ is defined by $\vw_t=P(\vn+t\vv)$ so that
\begin{equation}\label{4.3}
\vw_0(x)=\vn(x).
\end{equation}
Then
\begin{equation}\label{4.4}
\begin{split}
 \delta K(\vn)(\vv) &:= \left. \frac{d}{dt}K(\vn+t\vv) \right|_{t=0} \\ &=
 \left.\frac{d}{dt}
\int_\Omega \psi\left( \vn+t\vv\right) \left[ F\left( x,\vw_t,\nabla \vw_t \right) +\left| \vn+t\vv -\vw_t \right|^2 + \left| \nabla (\vn+t\vv)-\nabla \vw_t \right|^2  \right]\, dx\right|_{t=0} \\ &=
g'(0)+ 
\int_\Omega \left.\frac{d}{dt}\psi(\vn+t\vv)\right|_{t=0}(*)\,dx\\ &+
\left.\frac{d}{dt}\int_\Omega \psi\left( \vn+t\vv\right) \left[ \left| \vn+t\vv -\vw_t \right|^2 + \left| \nabla (\vn+t\vv)-\nabla \vw_t \right|^2  \right]\, dx\right|_{t=0} \\ &=
g'(0)+\int_\Omega \left.\frac{d}{dt}\psi(\vn+t\vv)\right|_{t=0} (**)+2\left.\psi(\vn+t\vv)(\vn+t\vv-\vw_t)\cdot (\vv-\vw'_t)\right|_{t=0} \\ &+
\left.2\psi(\vn+t\vv) \left( \nabla(\vn+t\vv)-\nabla \vw_t \right):\left( \nabla \vv-\nabla \vw'_t \right) \right|_{t=0}\, dx \\ &= 
g'(0) \\&=
0,
\end{split}
\end{equation}
where $(*)$ and $(**)$ denote other terms. When it comes to the second variation, we simply take one more derivative than \eqref{4.4} and find 
\begin{equation}\label{4.5}
\begin{split}
 &\delta^2 K(\vn)(\vv,\vv) \\ & =\left.\frac{d^2}{dt^2} K(\vn+t\vv) \right|_{t=0} \\ &=
  \left.\frac{d^2}{dt^2}
\int_\Omega \psi\left( \vn+t\vv\right) \left[ F\left( x,\vw_t,\nabla \vw_t \right) +\left| \vn+t\vv -\vw_t \right|^2 + \left| \nabla (\vn+t\vv)-\nabla \vw_t \right|^2  \right]\, dx\right|_{t=0} \\ &=
g''(0)+ \int_\Omega \left.\frac{d^2}{dt^2}\psi(\vn+t\vv)\right|_{t=0}(*)+\left.\frac{d}{dt}\psi(\vn+t\vv)\right|_{t=0}(**) \\ &+
\left.2\left[ |\vv-\vw'_t|^2-\vw''_t\cdot\left(\vn+t\vv-\vw_t\right)+
 |\nabla \vv -\nabla \vw'_t|^2-\left( \nabla(\vn+t\vv)-\nabla\vw_t \right):\nabla \vw''_t \right]\right|_{t=0}\,dx  \\ &=
g''(0)+2\int_\Omega |\vv-\vw'_0|^2+|\nabla\vv-\nabla\vw'_0|^2\,dx.
\end{split} 
\end{equation}

In both \eqref{4.4} and \eqref{4.5} we have been implicitly using the definition of $\psi$ to know that for any given $t$ small enough $\psi(\vn+t\vv)\equiv 1$. Now we use our assumption on the second variation of $I$ to show that the above expression is strictly positive.
\begin{equation}\label{4.6}
\begin{split}
 \delta^2K(\vn)(\vv,\vv) \geqslant& \int_\Omega \gamma|\vw'_0|^2+ \gamma|\nabla \vw'_0|^2+2|\vv-\vw'_0|^2+2|\nabla \vv-\nabla \vw'_0|^2\,dx. \\
 =&\frac{4}{2+\gamma}\int_\Omega \left| \vv-\left( 1+ \frac{\gamma}{2}\right)\vw'_0\right|^2+\left| \nabla \vv-\left( 1+ \frac{\gamma}{2}\right)\nabla\vw'_0\right|^2\,dx\\
 &+\frac{2\gamma}{2+\gamma}||\vv||_{1,2}^2 \\
 \geqslant & \frac{2\gamma}{2+\gamma} ||\vv||^{1,2}_2
 \end{split}
\end{equation}
This proves the positivity of the second variation of $K$. Therefore the well-known unconstrained result \cite[Ch 4,1.1]{giaquinta1996calculus} is applicable and implies that $\vn$ is a strict weak local minimizer of $K$. Thus there exists some $\epsilon>0$ such that 
\begin{equation}\label{4.7}
K(\vm)>K(\vn) \,\,\,\, \text{if} \,\,\,\, \vm \in \mathcal{B}\,\,\,\, \text{and} \,\,\,\, 0<||\vn-\vm||_{1,\infty} < \epsilon.
\end{equation}
This implies that if $\vm\in\mathcal{A}\subset \mathcal{B}$ and $0<||\vm-\vn||_{1,\infty}<\epsilon$, then
\begin{equation}\label{4.8}
I(\vm)=K(\vm)> K(\vn)=I(\vn).
\end{equation}
Therefore $\vn$ is a strict weak local minimizer of $I$.
\qed
\end{proof}

%
%
%
%
% Necessary conditions for a strong local minimizer
%
%
%
%

\section{Necessary conditions for a strong local minimizer}\label{sec:strong necessary}

Now we examine the extra conditions that a strong local minimizer must satisfy in addition to the conditions proved in section \ref{sec:weak necessary}. However in order to prove our constrained result, we need an unconstrained result which we can apply to our functional $J$. Husseinov \cite[Th. 1.8]{husseinov1995weierstrass} has proved such a result but its technicalities mean that it requires a degree of manipulation to be applicable in our situation. Hence we will present our own self-contained proof (following an unpublished argument of John Ball), which we can readily apply. Solely for the purposes of our next theorem the set of admissible functions will be
\begin{equation}\label{5.1}
 \mathcal{B}=\left.\left\{\, \vn\in W^{1,1}\left(\Omega,\mathbb{R}^k\right)\,\right|\,\vn=\vn_0\,\,\text{on}\,\,\partial\Omega_1\, \right\},
\end{equation}
with the slight alteration that $\vn_0\in W^{1,1}\left( \Omega,\mathbb{R}^k\right)$. We also do not require any differentiability of the Lagrangian in order to prove the quasiconvexity results. Hence just for the duration of the next three quasiconvexity proofs, we will only assume the Lagrangian is continuous in its arguments. So for Theorem \ref{quasiconvexityunconstrained} we assume that
\begin{equation}\label{5.2}
 F\in C\left( \overline{\Omega}\times\mathbb{R}^d\times\mathbb{R}^{k\times d}\right)
\end{equation}
and for Theorems \ref{quasiconvexity} and \ref{quasiconvexity-boundary} we assume that
\begin{equation}\label{5.3}
 F\in C\left( \overline{\Omega}\times M\times\mathbb{R}^{k\times d}\right).
\end{equation}

\begin{theorem}[Quasiconvexity in the interior - Unconstrained]\label{quasiconvexityunconstrained}
 Suppose that $\vn \in \mathcal{B} \cap W^{1,\infty}\left( \Omega,\mathbb{R}^k\right)$ is a strong local minimizer of $I$. Then for a.e. $x\in \Omega$
 \begin{equation}\label{5.4}
  F(x,\vn(x),\nabla\vn(x))\leqslant \frac{1}{|D|} \int_D F(x,\vn(x),\nabla\vn(x)+\nabla_y \vpsi(y))\,dy,
 \end{equation}
 for every open and bounded $D\subset \mathbb{R}^d$ and $\vpsi \in W^{1,\infty}_0\left(D,\mathbb{R}^k\right)$.
\end{theorem}

\begin{proof}
 
 We take an open and bounded $D\subset \mathbb{R}^d$, $\vpsi \in W^{1,\infty}_0\left(D,\mathbb{R}^k\right)$, and pick some $x\in \Omega$. We choose a $\delta>0$ such that $B(x,2\delta)\subset \Omega$ and define 
 \begin{equation}\label{5.5}
  \vphi^z_\epsilon(x):=\epsilon \vpsi\left( \frac{x-z}{\epsilon} \right),
 \end{equation}
 for some $z\in B(x,\delta)$. Then we can say that since $\vn$ is a strong local minimizer of $I$, if $\epsilon$ is sufficiently small, $\vphi^z_\epsilon \in C^\infty_0\left( \Omega,\mathbb{R}^k\right)$ and 
 \begin{equation}\label{5.6}
  I(\vn+\vphi^z_\epsilon)-I(\vn)\geqslant 0.
 \end{equation}
 Since this applies for each $z\in B(x,\delta)$ we can take some $f\in C^\infty_0\left( B(x,\delta)\right)$ which is non-negative to deduce
 \begin{equation}\label{5.7}
  0 \leqslant\int_{B(x,\delta)}f(z)\left[ I(\vn+\vphi^z_\epsilon)-I(\vn)\right]\,dz.
 \end{equation}
 Now we can change variables using the coordinate transformation
 \begin{equation}\label{5.8}
  y=\frac{x-z}{\epsilon},
 \end{equation}
 so that when we multiply through by $\epsilon^{-d}$, \eqref{5.7} becomes
 \begin{equation}\label{5.9}
 \begin{split}
  0\leqslant &\int_{B(x,\delta)} \int_D f(z) F(z+\epsilon y,\vn(z+\epsilon y)+\epsilon\vpsi(y),\nabla \vn (z+\epsilon y)+\nabla_y \vpsi(y))\,dy\,dz \\
  &-\int_{B(x,\delta)} \int_D f(z)F(z+\epsilon y,\vn(z+\epsilon y),\nabla \vn(z+\epsilon y))\,dy\,dz.
 \end{split}
 \end{equation}
From here we swap the integrals and perform another change of variables $z'=z+\epsilon y$ to find
\begin{equation}\label{5.10}
\begin{split}
 0\leqslant &\int_D \int_{B(x,\delta)+\epsilon y}f(z'-\epsilon y) F(z',\vn(z')+\epsilon \vpsi(y),\nabla \vn(z')+\nabla_y \vpsi(y))\,dz'\,dy \\
 &-\int_D \int_{B(x,\delta)+\epsilon y} f(z'-\epsilon y)F(z',\vn(z'),\nabla\vn(z'))\,dz'\,dy.
 \end{split}
\end{equation}

It is clear that we would like to use the Dominated Convergence Theorem on the integral in \eqref{5.10}. We see that this is possible by rewriting the integral using the indicator function $\mathbbm{1}_\epsilon(z'):=\mathbbm{1}_{B(x,\delta)+\epsilon y}(z')$ as
\begin{equation}\label{5.11}
\begin{split}
 &\int_D\int_{B(x,\delta)+\epsilon y}f(z'-\epsilon y) F(z',\vn(z')+\epsilon \vpsi(y),\nabla \vn(z')+\nabla_y \vpsi(y))\, dz'dy \\
 &-\int_D\int_{B(x,\delta)+\epsilon y}f(z'-\epsilon y) F(z',\vn(z'),\nabla\vn(z'))\,dz'\,dy\\
 =&\int_D\int_{B(x,2\delta)}f(z'-\epsilon y)F(z',\vn(z')+\epsilon \vpsi(y),\nabla \vn(z')+\nabla_y \vpsi(y))\mathbbm{1}_\epsilon (z')\,dz'\,dy \\
 &-\int_D\int_{B(x,2\delta)} f(z'-\epsilon y) F(z',\vn(z'),\nabla\vn(z'))\mathbbm{1}_\epsilon (z')\,dz'\,dy.
 \end{split}
\end{equation}
Since $f$ is smooth, $F$ is continuous and $\vn\in W^{1,\infty}\left( \Omega, M \right)$, we can easily take the limit of this as $\epsilon\rightarrow 0$ to see that it converges to
\begin{equation}\label{5.12}
 \int_D\int_{B(x,\delta)} f(z)\left[F(z,\vn(z),\nabla\vn(z)+\nabla_y \vpsi(y))-F(z,\vn(z),\nabla\vn(z))\right]\,dz\,dy.
\end{equation}
Writing this in an alternative fashion we find
\begin{equation}\label{5.13}
\begin{split}
 &\int_{B(x,\delta)}f(z)\int_D F(z,\vn(z),\nabla\vn(z)+\nabla_y \vpsi(y))-F(z,\vn(z),\nabla\vn(z))\,dydz \\
 =& \int_{B(x,\delta)}f(z)g(z)\,dz\geqslant0.
 \end{split}
\end{equation}
However since this is true for every non-negative $f\in C^\infty_0\left( B(x,\delta)\right)$, we conclude, using a standard mollification argument, that $g(z)\geqslant 0$ a.e. in $B(x,\delta)$. Since $x$ was arbitrary, this gives us the result.
\qed
\end{proof}

Now we can use Theorem \ref{quasiconvexityunconstrained} to prove the more general constrained quasiconvexity result. We return to our original set of admissible functions $\mathcal{A}$ with the boundary condition $\vn_0 \in W^{1,1}\left( \Omega,M\right)$ once more.

\begin{theorem}[Quasiconvexity in the interior]\label{quasiconvexity}
Suppose that $\vn \in \mathcal{A}\cap W^{1,\infty}\left(\Omega,M\right)$ is a strong local minimizer of $I$. Then for a.e. $x\in \Omega$
\begin{equation}\label{5.14}
 F\left(x,\vn(x),\nabla \vn(x) \right) \leqslant \frac{1}{|D|} \int_D F\left(x,\vn(x), \nabla \vn(x) +\nabla P(\vn(x)) \nabla \vphi(y) \right) \,dy,
\end{equation}
for any open and bounded $D\subset \mathbb{R}^{d}$, $\vphi \in W^{1,\infty}_0 \left( D,\mathbb{R}^{k} \right)$.
\end{theorem}

\begin{proof}

We take $D$ and $\vphi$ as given in the statement. We know $\vn$ is a strong local minimizer of $I$ and returning to the logic in \eqref{3.4} we have shown that if we take some $\vm \in \mathcal{B}$, then 
\begin{equation}\label{5.15}
||\vn-\vm||_{\infty}<\frac{\delta^*}{2} \quad \Rightarrow \quad P(\vm)\in\mathcal{A} \quad \text{and} \quad \left|\left| \vn-P(\vm) \right|\right|_\infty <2\left|\left| \vn-\vm \right|\right|_\infty
\end{equation}
Therefore if $\left|\left| \vn-\vm \right|\right|_\infty$ is small enough
\begin{equation}\label{5.16}
J(\vm)=I\left( P(\vm) \right) \geqslant I(\vn)=J(\vn),
\end{equation}
so $\vn$ is a strong local minimizer of $J$. Therefore we can apply Theorem \ref{quasiconvexityunconstrained} to deduce that for a.e. $x \in \Omega$
\begin{equation}\label{5.17}
 G\left(x,\vn(x),\nabla\vn(x)\right) \leqslant \frac{1}{|D|}\int_D G\left( x,\vn(x),\nabla\vn(x)+\nabla\vphi(y)\right)\,dy.
\end{equation}
To relate this back to $F$ we note that from the definition of $J$ we have 
\begin{equation}\label{5.18}
 G(x,\vn,{\bf Q})=\psi(\vn)F\left(x,P(\vn),\nabla P(\vn){\bf Q}\right).
\end{equation}
By substituting this back into we \eqref{5.17} we obtain
\begin{equation}\label{5.19}
 \begin{split}
  F\left( x,\vn(x),\nabla\vn(x)\right) &\leqslant \frac{1}{|D|}\int_\Omega \psi(\vn(x))F\left( x,\vn(x),\nabla P(\vn(x))(\nabla\vn(x)+\nabla \vphi(y)\right) \\
  &=\frac{1}{|D|} \int_D F\left(x,\vn(x), \nabla \vn(x) +\nabla P(\vn(x)) \nabla \vphi(y) \right) \,dy.
 \end{split}
\end{equation}
\qed
\end{proof}

\begin{theorem}[Quasiconvexity at the boundary]\label{quasiconvexity-boundary}
 Suppose that $\vn \in \mathcal{A}\cap C^{1}\left(\overline{\Omega},M\right)$ is a strong local minimizer of $I$. The exterior unit normal to $\Omega$ at a point $x$ is given by $\nu(x)$. Then
 \begin{equation}\label{5.20}
 F\left(x,\vn(x),\nabla \vn(x) \right) \leqslant \frac{1}{|{\bf B}_{\nu(x)}^-|} \int_{{\bf B}_{\nu(x)}^-} F\left(x,\vn(x), \nabla \vn(x) +\nabla P(\vn(x)) \nabla \vphi(y) \right) \,dy,
\end{equation}
for all $\vphi\in  \left\{ \,\vphi \in W^{1,\infty} \left( {\bf B}_{\nu(x)}^-,\mathbb{R}^k \right)\, \left|\,\vphi=0 \,\,\text{on}\,\, \partial {\bf B}\cap \partial {\bf B}_{\nu(x)}^-\, \right.\right\}$ where 
\begin{equation}\label{5.21}
 {\bf B}_{\nu(x)}^-=\left\{\,y\in{\bf B}\,|\,y\cdot\nu(x)<0\,\right\}.
\end{equation}
This holds for $\mathcal{H}^{d-1}$ a.e. $x\in \partial \Omega_2$ where $\partial\Omega_2$ is locally $C^1$.
\end{theorem}

\begin{remark}\label{remark4}
 Note that this condition only holds on the free boundary, $\partial\Omega_2$, of our problem and not on the entirety of $\partial \Omega$.
\end{remark}

\begin{proof}

In this instance we will apply the unconstrained result of Ball and Marsden \cite{ball1986quasiconvexity} to $J$. From the proof of Theorem \ref{quasiconvexity} we know that since $\vn$ is a strong local minimizer of $I$ it is also a strong local minimizer of $J$. Hence \cite{ball1986quasiconvexity} implies
\begin{equation}\label{5.22}
 G\left(x,\vn(x),\nabla \vn(x) \right) \leqslant \frac{1}{|{\bf B}_\nu^-|} \int_{{\bf B}_\nu^-} G\left(x,\vn(x), \nabla \vn(x) +\nabla \vphi(y) \right) \,dy,
\end{equation}
for $\mathcal{H}^{d-1}$ a.e. $x\in\partial\Omega_2$ and $\vphi$ as given in the statement. Now we can proceed with exactly the same logic as in the previous proof, using \eqref{5.18} and \eqref{5.19} to obtain the assertion.
\qed
\end{proof}

\vspace{3mm}

With our necessary quasiconvexity conditions in place we return to assuming that our Lagrangian is twice continuously differentiable in its second two arguments. The following corollary shows that quasiconvexity in the interior implies the more classical result of Weierstrass. This is intuitively understood from the fact that the Weierstrass condition is in reality a rank-one convexity condition in the interior which is weaker than quasiconvexity. This corollary motivates why in the next section we only need to strengthen the quasiconvexity conditions in order to prove the fundamental strong local minimum sufficiency result.

\begin{corollary}[Weierstrass Condition]\label{Weierstrass}
 Suppose that $\vn \in \mathcal{A}\cap W^{1,\infty}\left(\Omega,M\right)$ is a strong local minimizer of $I$. Then for a.e. $x\in \Omega$
 \begin{equation}\label{5.23}
  \mathcal{E}_F\left(x,\vn,\nabla \vn,\nabla\vn+\left[ \nabla P(\vn)\zeta\right]\otimes \eta\right) \geqslant 0
 \end{equation}
 for any $\zeta \in \mathbb{R}^k$ and $\eta \in \mathbb{R}^d$.
\end{corollary}
 
\begin{proof}
 
 By appealing to Theorem \ref{quasiconvexity} we know that \eqref{5.14} holds for a.e. $x\in\Omega$. Now for ${\bf Q}\in \mathbb{R}^{k\times d}$ we define
 \begin{equation}\label{5.24}
  \tilde{F}({\bf Q}):=F\left(x,\vn(x),\nabla P(\vn(x)){\bf Q} \right).
 \end{equation}
 Then using this new entity $\tilde{F}$, \eqref{5.14} simply says that 
 \begin{equation}\label{5.25}
  \frac{1}{|D|}\int_D \tilde{F}(\nabla\vphi(y))\,dy\geqslant \tilde{F}(0),
 \end{equation}
 for all $\vphi \in W^{1,\infty}_0\left(D,\mathbb{R}^k\right)$. Therefore $\tilde{F}$ is quasiconvex at $\nabla \vn(x)$ (for almost every $x$). We take some $\zeta \otimes \eta \in \mathbb{R}^{k\times d}$, then for $t\in [0,1)$, define
 \begin{equation}\label{5.26}
   A_1^t := \nabla \vn + (1-t) \zeta \otimes \eta, \quad
   A_2^t := \nabla \vn - t \zeta \otimes \eta.
 \end{equation}
 Then $tA_1^t +(1-t)A_2^t = \nabla \vn $ and $A_1^t - A_2^t = \zeta \otimes \eta$. We also know that quasiconvexity implies rank-one convexity \cite{dacorogna2008direct} so that from Definition \ref{def5}
 \begin{equation}\label{5.27}
  \tilde{F}(\nabla \vn)\leqslant t \tilde{F}(A_1^t)+(1-t)\tilde{F}(A_2^t)
 \end{equation}
 Hence by rewriting \eqref{5.27} and scaling $\zeta$, we find
 \begin{equation}\label{5.28}
  \tilde{F}(\nabla \vn) \leqslant t \tilde{F}(\nabla \vn + \zeta \otimes \eta) + (1-t) \tilde{F}\left( \nabla \vn - \frac{t}{1-t} \zeta \otimes \eta \right) .
 \end{equation}

To finish we notice that the right side of \eqref{5.28} is a continuously differentiable function in $t\in [0,1)$ which achieves its lower bound at $t=0$. Thus it has non-negative derivative at zero and we differentiate to find
\begin{equation}\label{5.29}
 \begin{split}
  0 & \leqslant \tilde{F}(\nabla \vn + \zeta\otimes \eta)-\tilde{F}(\nabla \vn) -\sum_{i,j} \tilde{F}_{{\bf Q}_{i,j}}(\nabla \vn)\zeta_i\eta_j \\
  &= \tilde{F}(\nabla \vn+\zeta\otimes \eta)-\tilde{F}(\nabla \vn)-\sum_{i,j}\left[ \sum_{\alpha,\beta} F_{\vn_{\alpha,\beta}}(x,\vn,\nabla\vn)P_{\alpha,i}(\vn)\delta_{\beta\,j} \right]\zeta_i\eta_j \\  &=
   F(x,\vn,\nabla\vn+\left[ \nabla P(\vn)\zeta\right]\otimes\eta)-F(x,\vn,\nabla\vn) 
   -F_{\nabla\vn}(x,\vn,\nabla\vn):\left[ \nabla P(\vn)\zeta\right]\otimes\eta,
 \end{split} 
\end{equation}
which is what we set out to prove. 
 \qed
\end{proof}

%
%
%
% Sufficient Conditions for a strong local minimizer
%
%
%

\section{Sufficient conditions for a strong local minimizer}\label{sec:strong sufficient}

Before we prove our strong sufficiency theorem using the result of Grabovsky \& Mengesha we prove a preliminary theorem based on a result by Taheri \cite{taheri2001sufficiency}. It shows that if we assume a very strong Weierstrass condition then there is almost no distinction between weak local minimizers and strong local minimizers. For these sufficiency results we need to assume some extra regularity conditions for our Lagrangian F, the first of which is a differentiability condition

\vspace{3mm}
\noindent
(L1) $F \in C^2\left(U \times O \times \mathbb{R}^{k\times d}\right)$ where $U$ and $O$ are open neighbourhoods of $\overline{\Omega}$ and $M$ respectively.

\begin{theorem}\label{Strong sufficient Taheri}
 
 Let $\vn \in\mathcal{A}\cap C^1\left(\overline{\Omega},M\right)$ and for $v\in \text{Var}_\mathcal{A}$ define (see \eqref{1.10})
 \begin{equation}\label{6.1}
  g(t)=I\left(P(\vn+t\vv)\right)=I(\vw_t).
 \end{equation}
 Suppose that the Lagrangian $F$ satisfies (L1) and there exists some $\gamma>0$ such that
 \begin{equation}\label{6.2}
  g'(0)=0\quad\text{and}\quad g''(0)\geqslant \gamma ||\vw'_0||_{1,2}^2\quad \forall \vv \in \text{Var}_{\mathcal{A}}.
 \end{equation}
 In addition, suppose that
 \begin{equation}\label{6.3}
  \mathcal{E}_F(x,\vm,\vQ_1,\vQ_2)\geqslant \gamma |\vQ_1-\vQ_2|^2
 \end{equation}
 for any $x \in \overline{\Omega}$, $\vm \in M$, and $\vQ_1,\vQ_2 \in \mathbb{R}^{k\times d}$. Then $\vn$ is a strict strong local minimizer of $I$.
 
 \end{theorem}
 
 \begin{remark}
  This theorem shows that any $C^1$ weak local minimizer is in fact a strong local minimizer if the Lagrangian is convex with respect to the gradient term.
 \end{remark}
 
 \begin{proof}
 
  As this is a sufficiency proof we will be using the functional $K$ as the unconstrained problem which is related to $I$. We will show that $K$ satisfies all of the required conditions to apply Taheri's sufficiency result \cite[Theorem 3.3]{taheri2001sufficiency}. Recall from the proof of Theorem \ref{Weak sufficient} that our assumptions immediately imply that 
  \begin{equation}\label{6.4}
   \delta K(\vn)(\vv)=0 \,\,\,\,\text{and}\,\,\,\, \delta^2 K(\vn)(\vv,\vv)\geqslant \frac{2\gamma}{2+\gamma} ||\vv||_{1,2}^2
  \end{equation}
  for every test function $\vv\in \text{Var}_{\mathcal{B}}$. To be able to apply Taheri's result we now need to show what strengthened Weierstrass condition is satisfied by $H$. We choose $0<\epsilon<\frac{\delta^*}{2}$ so that if $|\vm-\vn(x)|<\epsilon$ then $\psi(\vm)=1$. Then for any $x \in \overline{\Omega}$, $|\vm-\vn(x)|<\epsilon$ and $\vQ \in \mathbb{R}^{k\times d}$ we have 
  \begin{equation}\label{6.5}
   \begin{split}
    &\mathcal{E}_H\left(x,\vm,\nabla\vn(x),\vQ\right) \\
    =\quad&H(x,\vm,\vQ)-H(x,\vm,\nabla\vn)-H_{\nabla\vm}(x,\vm,\nabla\vn):(\vQ-\nabla\vn) \\ =\quad&
     F(x,P(\vm),\nabla P(\vm) \vQ)+|(I-\nabla P(\vm))\vQ|^2 \\ &
    -F(x,P(\vm),\nabla P(\vm)\nabla\vn)-|(I-\nabla P(\vm))\nabla \vn|^2 \\ &
    -F_{\nabla\vn}(x,P(\vm),\nabla P(\vm)\nabla \vn):\left(\nabla P(\vm)\vQ-\nabla P(\vm)\nabla \vn\right)\\&
    -2((I-\nabla P(\vm))^T(I-\nabla P(\vm))\nabla \vn):(\vQ-\nabla \vn) \\ =\quad&
     \mathcal{E}_F(x,\vm,\nabla P(\vm)\nabla \vn,\nabla P(\vm) \vQ)+ |(I-\nabla P(\vm))(\vQ-\nabla \vn)|^2 \\ \geqslant\quad&
     \gamma|\nabla P(\vm)(\vQ-\nabla \vn)|^2+|(I-\nabla P(\vm))(\vQ-\nabla \vn)|^2 \\ =\quad&
     \frac{\gamma}{\gamma+1} |\vQ-\nabla \vn|^2+\frac{1}{\gamma+1}|(I-(\gamma+1)\nabla P(\vm))(\vQ-\nabla \vn)|^2 \\ \geqslant\quad&
     \frac{\gamma}{\gamma+1} |\vQ-\nabla \vn|^2.
   \end{split}
  \end{equation}

  This means that we can apply \cite[Theorem 3.3]{taheri2001sufficiency} to deduce that there exist two constants, $\sigma_1,\sigma_2>0$ such that if $\vm \in \mathcal{B}$ with $||\vm-\vn||_\infty < \sigma_1$ then 
  \begin{equation}\label{6.6}
   K(\vm)-K(\vn) \geqslant \sigma_2 ||\vm-\vn||_{1,2}^2.
  \end{equation}
  Therefore $\vn$ is a strict strong local minimizer of $K$. As $\mathcal{A}\subset \mathcal{B}$ this means that $\vn$ is also a strict strong local minimizer of $I$.
  \qed
 \end{proof}

 \vspace{3 mm}

While Theorem \ref{Strong sufficient Taheri} is very useful, we do not want to restrict ourselves to just considering Lagrangians satisfying \eqref{6.3}. So to finish this section we will be applying the recently proved sufficiency result of Grabovsky \& Mengesha \cite{grabovsky2009sufficient} which requires much weaker quasiconvexity assumptions. However in order to use their result we need to set up our problem in a slightly different way and we refer the reader to their paper for full details. We have four conditions that we will assume our Lagrangian satisfies: the differentiability condition (L1) (see above), a growth condition, a coercivity condition and a uniform continuity condition.

\vspace{3 mm}
\noindent
(L2) There exists a constant $C>0$ and $p\geqslant 2$ such that for all ${\bf Q}\in\mathbb{R}^{k\times d}$ and $|\vn|=1$
\begin{equation}\label{6.7}
 |F(x,\vn,{\bf Q})|\leqslant C\left(1+|{\bf Q}|^p\right),
\end{equation}
\begin{equation}\label{6.8}
 |F_{\bf Q}(x,\vn,{\bf Q})|\leqslant C\left(1+|{\bf Q}|^{p-1}\right) \quad \text{and} \quad |F_{\vn}(x,\vn,{\bf Q})|\leqslant C\left(1+|{\bf Q}|^{p}\right).
\end{equation}

\vspace{3 mm}
\noindent
(L3) We assume that $F$ is bounded below, and that we have $C_1>0$ and $C_2>0$ such that
\begin{equation}\label{6.9}
 \int_\Omega F(x,\vm(x),\nabla\vm(x))\,dx \geqslant C_1||\vm||_{1,p}^p-C_2,
\end{equation}
for all $\vm\in \mathcal{A}$. If $p>2$, we assume in addition that there exists some $D_1>0$ and $D_2>0$, such that for every $\vm \in \mathcal{A}$ with $||\vm-\vn||_{\infty}\leqslant \frac{\delta^*}{2}$ we have
\begin{equation}\label{6.10}
 \int_\Omega F\left(x,\vm,\nabla\vm \right)-F(x,\vn,\nabla\vn)\,dx \geqslant D_1||\vm-\vn||_{1,p}^p-D_2||\vm-\vn||_{1,2}^2.
\end{equation}

\vspace{3 mm}
\noindent
(L4) For every $\epsilon>0$ there exists a $\delta>0$ such that for every ${\bf Q}\in \mathbb{R}^{k\times d}$, $|\vn|=1$, and $x_1,x_2\in \overline{\Omega}$ with $|x_1-x_2|\leqslant \delta$  
\begin{equation}\label{6.11}
 \frac{|F(x_1,\vn,{\bf Q})-F(x_2,\vn,{\bf Q})|}{1+|{\bf Q}|^p} < \epsilon.
\end{equation}

\begin{remark}\label{remark5}
 These assumptions are slightly different to those in \cite{grabovsky2009sufficient} but when we relate them to another functional they do come into line with those used in their result. In (L3) to avoid any issues with the projection map we assumed that $||\vm-\vn||_\infty<\frac{\delta^*}{2}$ but since we are investigating $L^\infty$ local behaviour this does not lose us any generality.
\end{remark}

\begin{theorem}\label{Strong sufficiency}
 
 Suppose that $\Omega$ is a $C^1$ bounded domain and the Lagrangian $F$ satisfies (L1)-(L4). Assume that $\vn\in \mathcal{A}\cap C^1\left( \Omega,M\right)$ with the following set of assumptions. There exists some $\gamma>0$ such that for all $\vv \in \text{Var}_{\mathcal{A}}$
 \begin{equation}\label{6.12}
  g'(0)=0\quad \text{and} \quad g''(0)\geqslant \gamma||\vw'_0||_{1,2}^2.
 \end{equation}
 Suppose that for every open, bounded set $D\subset \mathbb{R}^d$, $\vphi \in W^{1,\infty}_0\left( D,\mathbb{R}^k \right)$ and $x\in\Omega$
 \begin{equation}\label{6.13}
 \begin{split}
  \int_D \mathcal{E}_F\left( x,\vn(x),\nabla\vn(x),\nabla\vn(x)+\nabla P(\vn(x))\nabla\vphi(y) \right)  \,dy  & \\ \geqslant \gamma\int_D \left| \nabla P(\vn)\nabla\vphi(y) \right|^2\,dy.
 \end{split}
 \end{equation}
 Suppose also that
 \begin{equation}\label{6.14}
 \begin{split}
  \int_{{\bf B}_\nu^-}\mathcal{E}_F\left( x,\vn(x),\nabla\vn(x),\nabla\vn(x)+\nabla P(\vn(x))\nabla\vphi(y) \right) \,dy &\\ \geqslant \gamma \int_{{\bf B}_\nu^-} \left| \nabla P(\vn)\nabla\vphi(y) \right|^2\,dy
 \end{split}
 \end{equation}
 for all $\vphi\in \left\{ \,\vphi \in W^{1,\infty} \left( {\bf B}_{\nu(x)}^-,\mathbb{R}^k \right)\,\left|\,\vphi=0 \,\,\text{on}\,\, \partial {\bf B}\cap \partial {\bf B}_{\nu(x)}^-\,\right.\right\}$ and  $ x\in \partial \Omega_2$, where ${\bf B}_{\nu(x)}^-$ is given in \eqref{5.21}. Then $\vn$ is a strong local minimizer of $I$.
 \end{theorem}
 
 \begin{proof}
 
 Unfortunately in this proof we cannot simply use the functional $K$ to apply \cite[Theorem 5.1]{grabovsky2009sufficient} and then relate the condition back to $I$. This is because $K$ does not satisfy the required coercivity conditions unless $p=2$. Therefore we introduce and study a related functional $\tilde{K}$. If $p=2$ we let $\tilde{K}=K$, otherwise we define
 
 \begin{equation}\label{6.15}
 \begin{split}
  \tilde{K}(\vm) &:=\int_\Omega \tilde{H}(x,\vm,\nabla\vm)\,dx \\&=K(\vm)+\int_\Omega \psi(\vm)\left( \left| \vm-P(\vm) \right|^p+\left| \nabla\vm-\nabla\left[P(\vm)\right] \right|^p \right)\,dx.
 \end{split}
 \end{equation}
 In other words this functional has been given a form of stability with respect to the $W^{1,p}$ norm as well as the $W^{1,2}$ norm. The first thing to note about this functional $\tilde{K}$ is that
 \begin{equation}\label{6.16}
  \delta \tilde{K}(\vn)(\vv)=\delta K(\vn)(\vv)\quad \text{and} \quad \delta^2 \tilde{K}(\vn)(\vv,\vv)=\delta^2 K(\vn)(\vv,\vv) \quad \forall \,\,\vv\in \text{Var}_{\mathcal{A}}.
 \end{equation}
 Therefore when we take an arbitrary $\vv \in \text{Var}_{\mathcal{A}}$, using the same logic as from the proof of the weak sufficiency theorem, we deduce
 \begin{equation}\label{6.17}
  \delta \tilde{K}(\vn)(\vv)=0 \quad \text{and} \quad \delta^2 \tilde{K}(\vn)(\vv,\vv)\geqslant \frac{2\gamma}{2+\gamma} ||\vv||_{1,2}^2.
 \end{equation}

 We also need to show that the strengthened quasiconvexity conditions used in \cite[Theorem 5.1]{grabovsky2009sufficient} hold for $\tilde{K}$, given that $F$ satisfies \eqref{6.13} and \eqref{6.14}. This is achieved with a careful calculation, if $p>2$ then
 \begin{equation}\label{6.18}
 \begin{split}
 &\mathcal{E}_{\tilde{H}}\left( x,\vn,\nabla\vn,\nabla\vn+\nabla\vphi(y) \right)  \\
 =\quad & \tilde{H}(x,\vn,\nabla \vn+\nabla \vphi(y))-\tilde{H}(x,\vn,\nabla \vn)- \nabla\vphi(y):\tilde{H}_{\nabla\vn}(x,\vn,\nabla\vn) \\
 =\quad & F\left(x,\vn,\nabla\vn+\nabla P(\vn)\nabla\vphi(y)\right)+\left|\nabla \vphi(y)-\nabla P(\vn)\nabla \vphi(y)\right|^2-F(x,\vn,\nabla\vn) \\
 &+\left|\nabla \vphi(y)-\nabla P(\vn)\nabla \vphi(y)\right|^p-\nabla P(\vn)\nabla\vphi(y):F_{\nabla\vn}\left(x,\vn,\nabla\vn\right) \\
 \geqslant\quad& \mathcal{E}_F \left(x,\vn,\nabla\vn,\nabla\vn+\nabla P(\vn)\nabla\vphi(y) \right)+\left|\nabla \vphi(y)-\nabla P(\vn)\nabla \vphi(y)\right|^2,
 \end{split}
 \end{equation}
 and the $p=2$ case is the same but without the term with the $p^\text{th}$ power. Now we can apply our strengthened quasiconvexity condition to \eqref{6.18} to find
 \begin{equation}\label{6.19}
 \begin{split}
   &\int_D \mathcal{E}_{\tilde{H}}\left(x,\vn(x),\nabla\vn(x),\nabla\vn(x)+\nabla\vphi(y) \right)\,dy \\
   \geqslant& \int_D \gamma  |\nabla P(\vn)\nabla\vphi(y)|^2+\left|\left(Id-\nabla P(\vn)\right) \nabla\vphi(y) \right|^2\,dy \\ 
    =& \int_D \gamma |\nabla P(\vn) \nabla \vphi(y)|^2+|\nabla\vphi(y)|^2dy \\ 
    & -\int_D 2\nabla P(\vn)\nabla\vphi(y):\nabla\vphi(y)+\left| \nabla P(\vn)\nabla\vphi(y)\right|^2 \,dy \\
    = & \int_D \frac{\gamma}{1+\gamma} |\nabla\vphi(y)|^2 +\frac{1}{1+\gamma}\left| \nabla\vphi(y)-(1+\gamma)\nabla P(\vn)\nabla\vphi(y) \right|^2\,dy \\
    \geqslant& \frac{\gamma}{1+\gamma}\int_D |\nabla \vphi(y)|^2\,dy
 \end{split}
 \end{equation}
 for any open, bounded $D$ and $\vphi \in W^{1,\infty}_0\left(D,\mathbb{R}^k\right)$. The same is also true of the quasiconvexity at the boundary. Therefore the final thing to consider before we can apply \cite[Theorem 5.1]{grabovsky2009sufficient} are the conditions which $\tilde{H}$ must satisfy as the counterparts to (L1)-(L4). From the definition of $K$, the regularity assumption (L1) implies that
 \begin{equation}\label{6.20}
 K\in C^2\left(U\times \mathbb{R}^k \times \mathbb{R}^{k\times d}\right),
 \end{equation}
 and the additional terms for $\tilde{K}$ are clearly twice differentiable so that $\tilde{K}$ also satisfies \eqref{6.20}. Equally it is clear to see that the growth condition (L2) must also hold for $\tilde{K}$. For the uniform continuity condition we can show that $\tilde{H}$ satisfies a slightly different condition to (L4). 
 
 \vspace{3mm}
 \noindent
 (L4*) For every $R>0$ and $\epsilon>0$ there exists a $\delta>0$ such that for every ${\bf Q}\in \mathbb{R}^{k\times d}$, $|\vn|<R$, and $x_1,x_2\in \overline{\Omega}$ with $|x_1-x_2|\leqslant \delta$  
\begin{equation}\label{6.21}
 \frac{|\tilde{H}(x_1,\vn,{\bf Q})-\tilde{H}(x_2,\vn,{\bf Q})|}{1+|{\bf Q}|^p} < \epsilon.
\end{equation}

In the language of Grabovsky \& Mengesha's paper this means that $\tilde{H}$ satisfies conditions (H1), (H2) and (H4) so that we only need to deal with the coercivity condition in order to apply their result. They use the coercivity condition to show that the problem of strong local minimizers can effectively be reduced to $W^{1,p}$ local minimizers. We will prove this directly with our (L3) assumption mimicking the reasoning in \cite[Section 7]{grabovsky2009sufficient}. On a technical note we choose $\epsilon>0$ to be a fixed constant, sufficiently small, such that if $\vm \in\mathcal{B}$ and $||\vm-\vn||_\infty<\epsilon$ then 
\begin{equation}\label{6.22}
 \max\left\{ \epsilon,\left|\left| \vn-P(\vm) \right|\right|_\infty \right\} \leqslant \frac{\delta^*}{2}.
\end{equation}
When we rewrite \eqref{6.9} and \eqref{6.10} as conditions on $\tilde{K}$ we see that if we take an $\vm \in \mathcal{B}$ with $||\vm-\vn||_{\infty}<\epsilon$ 
 \begin{equation}\label{6.23}
  \tilde{K}(\vm)\geqslant C_1 \left|\left| P(\vm) \right|\right|_{1,p}^p-C_2+\left|\left|\vm-P(\vm)\right|\right|_{1,2}^2+\left|\left|\vm-P(\vm)\right|\right|_{1,p}^p
 \end{equation}
 and if $p=2$ we do not have the final term in the equation above. As for the additional assumption \eqref{6.10}, if we take $\vm\in \mathcal{B}$ with $||\vn-\vm||_\infty <\epsilon$, then
 \begin{equation}\label{6.24}
 \begin{split}
 \tilde{K}(\vm)-\tilde{K}(\vn)  \geqslant & D_1 \left|\left|\vn-P(\vm)\right|\right|_{1,p}^p-D_2\left|\left|\vn-P(\vm)\right|\right|_{1,2}^2\\
 &+\left|\left|\vm-P(\vm)\right|\right|_{1,2}^2 + \left|\left|\vm-P(\vm)\right|\right|_{1,p}^p.
 \end{split}
 \end{equation}
 Here we need to introduce a little notation from \cite{grabovsky2009sufficient} to follow their reasoning.
 \begin{definition}\label{def6}
  We say that a sequence $\left( \vm_j\right) \subset {\mathcal{B}}$ is a strong variation if $||\vm_j-\vn||_{\infty} \rightarrow 0$ as $j\rightarrow \infty$.
 \end{definition}
 
 To avoid any technical difficulties we will always assume that $||\vm_j-\vn||_{\infty}<\epsilon$ for all $j$ for any strong variation. With this definition Grabovsky \& Mengesha look to prove that the normalized increment
 \begin{equation}\label{6.25}
  \liminf_{j\rightarrow \infty}\frac{\tilde{K}(\vm_j)-\tilde{K}(\vn)}{||\vm_j-\vn||_{1,2}^2}\geqslant 0
 \end{equation}
 for any strong variation $\left(\vm_j\right)$. The estimate \eqref{6.23} means that if we take a strong variation $\left( \vm_j \right)$ such that $||\nabla \vm_j||_{p}\rightarrow \infty$, we automatically have
 \begin{equation}\label{6.26}
  \tilde{K}(\vm_j) \geqslant \tilde{K}(\vn) \quad \forall j\geqslant N,
 \end{equation}
 for some $N>0$, which means that the normalized increment condition \eqref{6.25} is satisfied. Therefore we only need to consider strong variations $\left( \vm_j\right)$ such that $\vm_j \rightharpoonup \vn$ in $W^{1,p}$. Let
 \begin{equation}\label{6.27}
  \alpha_j =||\vm_j-\vn||_{1,2} \leqslant D ||\vm_j-\vn||_{1,p}=\beta_j,
 \end{equation}
 and we will show that \eqref{6.25} holds unless
 \begin{equation}\label{6.28}
  \lim_{j \rightarrow \infty} \alpha_j=\lim_{j \rightarrow \infty} \beta_j =0 \quad \text{and} \quad \lim_{j \rightarrow \infty} \frac{ \beta_j^p}{\alpha_j^2} <\infty.
 \end{equation}

 Since $\vm \rightharpoonup \vn$ in $W^{1,p}$ we know that $\left(\alpha_j\right)$ and $\left(\beta_j\right)$ are bounded sequences and without loss of generality we assume that $\alpha_j \rightarrow \alpha_0>0$. Let $Q\tilde{H}(x,\vn,{\bf Q})$ be the quasiconvexification of $\tilde{H}(x,\vn,{\bf Q})$ with respect to ${\bf Q}$, and we exploit the fact that the quasiconvexified functional is $W^{1,p}$ sequentially weakly lower semicontinuous \cite[Theorem 8.11]{dacorogna2008direct} with \eqref{6.13} to obtain
 \begin{equation}\label{6.29}
 \begin{split}
  \liminf_{j \rightarrow \infty} \frac{ \tilde{K}(\vm_j)-\tilde{K}(\vn)}{\alpha_j^2} &= \frac{1}{\alpha_0^2} \liminf_{j \rightarrow \infty} \int_\Omega \tilde{H}(x,\vm_j,\nabla\vm_j)-\tilde{H}(x,\vn,\nabla\vn)\,dx \\ & \geqslant
  \frac{1}{\alpha_0^2}\liminf_{j \rightarrow \infty}\int_\Omega Q\tilde{H}(x,\vm_j,\nabla\vm_j)-\tilde{H}(x,\vn,\nabla\vn)\,dx \\ & \geqslant
  \frac{1}{\alpha_0^2}\int_\Omega Q\tilde{H}(x,\vn,\nabla\vn)-\tilde{H}(x,\vn,\nabla\vn)\,dx \\ &=0.
 \end{split}
 \end{equation}
 Therefore we need only look at the cases where $\alpha_j \rightarrow 0$ which is all we require if $p=2$. A corollary of the fact that $\alpha_j \rightarrow 0$ for a strong variation $\left(\vm_j \right)$ is that 
 \begin{equation}\label{6.30}
  P\left(\vm_j\right)\rightarrow \vn \quad \text{in} \quad W^{1,2}\left(\Omega,\mathbb{R}^k \right).
 \end{equation}
 If $p>2$, this corollary, together with \eqref{6.24}, shows us that if $\alpha_j \rightarrow 0$, the normalized increment from \eqref{6.25} is automatically non-negative unless
 \begin{equation}\label{6.31}
  \vn-P\left( \vm_j\right)\rightarrow 0 \quad \text{and} \quad \vm_j-P\left( \vm_j\right)\rightarrow 0 \quad \text{in} \quad W^{1,p} \,\, \Rightarrow \,\, \vm_j\rightarrow \vn \quad \text{in} \quad W^{1,p} 
 \end{equation}
 This means that we have reduced the problem to $W^{1,p}$ local minimizers for $\tilde{K}$, exactly as in \cite[Section 7]{grabovsky2009sufficient}. Now we have amassed all of these conditions and dealt with the (L3) assumption, we are in a position to apply \cite[Theorem 5.2]{grabovsky2009sufficient} and its direct corollary \cite[Theorem 5.1]{grabovsky2009sufficient}. This tells us that $\vn$ is a strong local minimizer of $\tilde{K}$ which implies that there exists some $ \delta >0$ such that if $\vm\in\mathcal{A}\subset \mathcal{B}$ such that $||\vm-\vn||_\infty <\delta$ then
\begin{equation}\label{6.32}
 I(\vm)=\tilde{K}(\vm)\geqslant \tilde{K}(\vn)=I(\vn).
\end{equation}
Thus $\vn$ is a strong local minimizer of $I$.
\qed
\end{proof}

%
%
%
%
% Specific examples 
%
%
%
%

\section{Illustrative applications}\label{sec:applications}

In this final section we will apply the results we have established to rigorously investigate stability for some constrained problems. The motivation for this work came from liquid crystals so it is logical to begin there. We recall from the introduction that the standard one-constant, cholesteric Oseen-Frank theory problem is to minimize
\begin{equation}\label{7.1}
 I(\vn)=K\int_\Omega |\nabla \vn|^2+2t\, \vn \nabla\times \vn +t^2\, dx
\end{equation}
where $K,t\geqslant0$ are constants, $\Omega=[-L_1,L_1]\times[-L_2,L_2]\times[0,1]$, and the set of admissible mappings is 
\begin{equation}\label{7.2}
 \mathcal{A}:=\left.\left\{ \, \vn\in W^{1,2}\left( \Omega,\mathbb{S}^2\right)\,\right|\, \vn|_{z=0}=\vn|_{z=1}={\bf e}_3, \,\, \vn|_{x=-L_1} = \vn|_{x=L_1},\,\, \vn|_{y=-L_2}=\vn|_{y=L_2}\,\right\}.
\end{equation}
This situation clearly falls within the scope of Sections \ref{sec:weak necessary}-\ref{sec:strong sufficient}. The boundary conditions we impose are frustrated because it discounts the natural structure that a cholesteric liquid crystal prefers to form. The basic cholesteric helical configuration is given by
\begin{equation}\label{7.3}
 \vn(z)=\left( \begin{array}{c} \cos(tz) \\ \sin(tz) \\ 0 \end{array} \right).
\end{equation}
Due to this frustration, we will consider the stability of the constant function $\vn={\bf e}_3$, also called the unwound state, as the parameter $t$ changes.
This question has been investigated with experiments \cite{gartland2010electric} and simulations before but does not seem to have been treated within an analytical mathematical framework. It is a simple exercise to show that the strong form of the Euler-Lagrange equation for \eqref{7.1} is
\begin{equation}\label{7.4}
 \Delta \vn - 2t\nabla \times \vn + \vn\left( |\nabla \vn|^2+2t\,\vn\nabla\times \vn\right)=0.
\end{equation}
Clearly $\vn={\bf e}_3$ is always a solution of this equation and the following proposition quantifies its stability.

\begin{proposition}\label{LCprop}
Consider the variational problem as described in \eqref{7.1} and \eqref{7.2}. If $t<\pi$ then $\vn={\bf e}_3$ is a strict strong local minimizer of $I$. If $t>\pi$ then $\vn={\bf e}_3$ is not a weak local minimizer of $I$.
\end{proposition}

\begin{proof}

The heart of the proof centers around showing that the second variation of $I$ can be negative if $t>\pi$ and proving its positivity when $t<\pi$. Therefore when $t<\pi$ we are looking to show that for some $\delta>0$
 \begin{equation}\label{7.5}
  \left.\frac{d^2}{d\epsilon^2}I\left( \frac{\vn+\epsilon \vv}{|\vn+\epsilon \vv|} \right) \right|_{\epsilon=0} = \left.\frac{d^2}{d\epsilon^2}I\left(\vw_\epsilon\right)\right|_{\epsilon=0} \geqslant \delta ||\vw'(0)||_{1,2}^2
 \end{equation}
 for every $\vv \in \text{Var}_\mathcal{A}$ where 
 \begin{equation}\label{7.6}
 \text{Var}_\mathcal{A}:=\left\{ \, \vv\in C^\infty\left( {\Omega},\mathbb{R}^3\right)\,|\, \vv|_{z=0}=\vv|_{z=1}=0,\,\, \vv|_{x=-L_1}=\vv|_{x=L_1},\,\, \vv|_{y=-L_2} = \vv|_{y=L_2}\,\right\}.
\end{equation}
By referring to Theorem \ref{Strong sufficient Taheri}, we see that \eqref{7.5} is sufficient to show that the state is a strict strong local minimizer as the strengthened Weierstrass condition \eqref{6.3} is automatically satisfied. To accomplish this we need to find the exact form of the second variation. Let $\vv \in \text{Var}_\mathcal{A}$ and $\vw_\epsilon:=\frac{\vn+\epsilon \vv}{|\vn+\epsilon \vv|}$, then

\begin{equation}\label{7.7}
\begin{split}
 \left.\frac{d^2}{d\epsilon^2}I\left( \vw_\epsilon \right)\right|_{\epsilon=0} =& 2\int_\Omega |\nabla \vw'_0|^2+2t\vw'_0\cdot\nabla\times \vw'_0\,dx \\&+2\int_\Omega\nabla \vw''_0:\nabla \vw_0+2t\vw''_0\cdot \nabla \times \vw_0\,dx \\  
 =& 2\int_\Omega |\nabla \vw'_0|^2+2t\vw'_0\cdot\nabla\times \vw'_0\,dx \\
 =& 2\int_\Omega |\nabla \vv_1|^2+|\nabla\vv_2|^2+2t\left(\vv_2\vv_{1,3}-\vv_1 \vv_{2,3}\right)\, dx.
 \end{split}
\end{equation}

Note that only the derivatives in $z$ can possibly make \eqref{7.7} negative since the other derivatives only appear as perfect squares. Therefore we proceed initially by proving that ${\bf e}_3$ is a strict weak local minimizer of $I$ over functions which just depend on $z$. Let
 \begin{equation}\label{7.8}
   \vn(z) = \left( \begin{array}{c} \cos\theta\cos\phi \\\cos\theta\sin\phi \\ \sin \theta \end{array} \right)\,\,\text{where}\,\, \phi=\phi(z),\,\,\text{and}\,\, \theta=\theta(z),
 \end{equation}
 then we can perform the following energy estimate
 \begin{equation}\label{7.9}
 \begin{split}
  I(\vn)-I({\bf e}_3) &=4L_1L_2\int_0^1 \theta'^2 +2t\cos^2\theta\left( \phi'^2-\phi'\right)\,dx \\&\geqslant 4L_1L_2 \int_0^1 \theta'^2-t^2 \cos^2\theta\,dz.
 \end{split}
 \end{equation}
 We let 
 \begin{equation}\label{7.10}
  J(\theta) := \int_0^1 \theta'^2-t^2 \cos^2\theta\,dz,
 \end{equation}
 and we will show that $\theta=\frac{\pi}{2}$ is a strict weak local minimizer of $J$ over 
 \begin{equation}\label{7.11}
  \mathcal{C}:=\left\{ \left.\, f\in W^{1,2}\left(0,1\right)\,\right|\,f(0)=f(1)=\frac{\pi}{2}\,\right\}
 \end{equation}
 when $t<\pi$ and not a weak local minimizer when $t>\pi$. We prove the stability first. The constant $\frac{\pi}{2}$ clearly satisfies the Euler Lagrange equation for $J$ and when we calculate the second variation we find 
 \begin{equation}\label{7.12}
  \left.\frac{d^2}{d\epsilon^2} J\left(\frac{\pi}{2}+\epsilon g\right)\right|_{\epsilon=0} =2 \int_0^1 g'^2-t^2 g^2\,dz.
 \end{equation}
 However we know that our test space is $W^{1,2}_0(0,1)$ and the Poincar\'{e} constant for those functions is precisely $\pi^2$ \cite{dym1972fourier}. Therefore if $t<\pi$
 \begin{equation}\label{7.13}
  \left.\frac{d^2}{d\epsilon^2} J\left(\frac{\pi}{2}+\epsilon g\right)\right|_{\epsilon=0}\geqslant \frac{\left(\pi^2-t^2\right)}{\pi^2} \int_0^1 g'^2\,dz \geqslant \alpha ||g||_{1,2}^2.
 \end{equation}
 
 The standard 1-dimensional theory therefore implies that $\theta=\frac{\pi}{2}$ is a strict weak local minimiser of $J$. Now we take some $\vv = \left( \vv_1, \vv_2 , 0 \right)\in \text{Var}_\mathcal{A}$ such that $\vv=\vv(z)$ and $||\vv||_{1,2}>0$. A small calculation, using a Taylor series expansion, shows that the Euler angle $\theta_\epsilon$ associated with $\vw_\epsilon$ is given by
 \begin{equation}\label{7.14}
  \theta_\epsilon = \sin^{-1} \left( \frac{1}{\left( 1+\epsilon^2 (\vv_1^2+\vv_2^2) \right)^\frac{1}{2}} \right) = \frac{\pi}{2} - (\vv_1^2+\vv_2^2)^\frac{1}{2}\epsilon + \epsilon^3 f_\epsilon,
 \end{equation}
 where $||f_\epsilon ||_{1,\infty} \leqslant C$ uniformly around $\epsilon = 0$. This immediately implies that $\theta_\epsilon \rightarrow \frac{\pi}{2}$ in $W^{1,\infty}$. Thus for all $\epsilon$ sufficiently small, $I(\vw_\epsilon) - I({\bf e}_3) \geqslant J(\theta_\epsilon) > 0$. As a result of these inequalities we know that
 \begin{equation}\label{7.15}
 \left. \frac{d^2}{d\epsilon^2}I(\vw_\epsilon)\right|_{\epsilon=0} \geqslant \left. \frac{d^2}{d\epsilon^2}J(\theta_\epsilon)\right|_{\epsilon=0}.
 \end{equation}
 When we combine \eqref{7.15} with the alternative formulation of $\theta_\epsilon$ we deduce
 \begin{equation}\label{7.16}
 \left. \frac{d^2}{d\epsilon^2}J(\theta_\epsilon)\right|_{\epsilon=0} = \left. \frac{d^2}{d\epsilon^2}J\left( \frac{\pi}{2} + \epsilon g + \epsilon^3 f_\epsilon \right)\right|_{\epsilon=0}= \left. \frac{d^2}{d\epsilon^2}J\left( \frac{\pi}{2} + \epsilon g  \right)\right|_{\epsilon=0}>0.
\end{equation}
Therefore we have proved that for any $\vv(z) \in \text{Var}_\mathcal{A}$, the second variation is strictly positive so long as $\vw'_0$ is non-zero. Returning to the form of the second variation \eqref{7.7} we see that the perfect squares in the $x$ and $y$ variables imply that if we now take an arbitrary $\vv\in \text{Var}_\mathcal{A}$ such that $||\vw'_0||_{1,2}>0$, then  
 \begin{equation}\label{7.18}
  \left. \frac{d^2}{d\epsilon^2}I(\vw_\epsilon)\right|_{\epsilon=0}>0.
 \end{equation}

To conclude stability we just need to show that \eqref{7.18} is in fact bounded below by $\delta||\vw'(0)||_{1,2}^2$. We argue by contradiction. If this was not the case then we could find a sequence of functions $\left( \vv^j\right)$ such that
 \begin{equation}\label{7.19}
||\vw'_{0 j}||_{1,2}^2=\left|\left|\,\left( \begin{array}{c} \vv^j_1 \\ \vv^j_2 \\ 0\end{array} \right)\,\right|\right|_{1,2}^2=1,
 \end{equation}
 and the second variation about these functions tended to zero. Without loss of generality we can assume that $\vv^j_1\rightharpoonup \vv_1$ and $\vv^j_2\rightharpoonup \vv_2$ in $W^{1,2}\left(\Omega\right) \subset\subset L^2\left(\Omega\right)$ so that when we take the $\liminf$ of the second variation we find
 \begin{equation}\label{7.20}
 \begin{split}
  0 &= \liminf_{j\rightarrow \infty} \int_\Omega |\nabla \vv_1^j|^2+|\nabla \vv_2^j|^2+2t\left( -\vv^j_1\vv^j_{2,3}+\vv^j_{1,3}\vv^j_2\right)\,dx \\
  &\geqslant \int_\Omega |\nabla \vv_1|^2+|\nabla \vv_2|^2+2t\left( -\vv_1\vv_{2,3}+\vv_{1,3}\vv_2\right)\,dx \\
  &\geqslant 0.
 \end{split}
 \end{equation}
Hence $\vv_1=\vv_2=0$ and so to have the second variation tend to zero we require that 
\begin{equation}\label{7.21}
 \int_\Omega |\nabla\vv^j_1|^2+|\nabla\vv^j_2|^2\,dx \rightarrow 0,
\end{equation}
and this clearly contradicts \eqref{7.19}. Therefore we find
 \begin{equation}\label{7.22}
  \left.\frac{d^2}{d\epsilon^2}I\left( \vw(\epsilon) \right)\right|_{\epsilon=0}>\delta ||\vw'(0)||_{1,2}^2
 \end{equation}
 for some $\delta>0$. So Theorem \ref{Strong sufficient Taheri} implies ${\bf e}_3$ is a strict strong local minimizer of $I$ if $t<\pi$. For the instability we return to the second variation of the functional $J$ as given in \eqref{7.12}. By setting $g(z)=\sin(\pi z)$ we see that the second variation is negative if $t>\pi$. Therefore there exists a sequence $\left( \theta_j\right) \subset \mathcal{C}$, such that $\theta_j \rightarrow \frac{\pi}{2}$ in $W^{1,\infty}$ and $J(\theta_j)<0$ for all $j$. Define 
 \begin{equation}\label{7.23}
  \vm_j(z) :=\left( \begin{array}{c} \cos \theta_j \cos(tz) \\ \cos \theta_j \sin(tz) \\ \sin \theta_j \end{array} \right)
 \end{equation}
 and we will show an upper bound for $||\vm_j-{\bf e}_3||_{1,\infty}$. The mean value theorem allows us to say that
 \begin{equation}\label{7.24}
 \begin{split}
  |\vm_j -{\bf e}_3|^2 & = 2 ( 1-\sin \theta_j ) \\
  & \leqslant C_1\left|\theta_j -\frac{\pi}{2}\right|.
  \end{split}
 \end{equation}
 Similarly 
 \begin{equation}\label{7.25}
 |\nabla\vm_j|^2=\theta_j'^2+ t^2\cos^2 \theta_j \leqslant \theta_j'^2+C_1t^2\left|\theta_j-\frac{\pi}{2}\right|,
 \end{equation}
 therefore $||\vm_j - {\bf e}_3||_{1,\infty} \rightarrow 0$ and 
 \begin{equation}\label{7.26}
  I(\vm_j)-I({\bf e}_3) = J(\theta_j) <0.
 \end{equation}
 This proves that ${\bf e}_3$ cannot be a weak local minimizer when $t>\pi$.

 \qed
\end{proof}

\begin{remark}
 It is not clear whether the constant state ${\bf e}_3$ is a local minimizer or not at the critical value of $t=\pi$. Although the problem is smooth in the twist parameter $t$, unless we know that ${\bf e}_3$ is a global minimizer for $t<\pi$ we do not know of a way to infer its stability at the critical value.
\end{remark}

The constrained results presented in sections \ref{sec:weak necessary}-\ref{sec:strong sufficient} roughly mimic those of the classical problems. However not all ideas from the classical study of Calculus of Variations translate to their constrained counterparts. In the next example we will use the circle $\mathbb{S}^1$ as our target space. Then we show that having a Lagrangian which is convex in the gradient does not guarantee a weak local minimizer is actually a global minimizer. In fact we go further and show that the problem has countably many strong local minimizers which are not global minima. Results and ideas of this kind are already known but have not been approached from this angle to the best of our knowledge. Consider the simplest one dimensional functional
 \begin{equation}\label{7.27}
  I(\vn)=\int_0^1 |\vn_x|^2 \,dx,
 \end{equation}
 over the set of admissible functions 
 \begin{equation}\label{7.28}
  \mathcal{A}:=\left\{ \left.\, \vn \in W^{1,2}\left( (0,1),\mathbb{S}^1 \right) \,\right|\,\vn(0)=\vn(1)={\bf e}_1 \, \right\}.
 \end{equation}
 
 It is a simple exercise (see \cite[p. 496]{evans2010partial}) to show that the weak form of the Euler-Lagrange equation for this one variable problem is 
 \begin{equation}\label{7.29}
 \int_0^1  \vv_x \cdot \vn_x-\left( \vn \cdot \vv\right) |\vn_x|^2 \,dx=0,
 \end{equation}
 for all $\vv \in C^\infty_0\left( (0,1),\mathbb{R}^2\right)$. It is also easy enough to realize that we have an infinite number of solutions to this differential equation by noticing that 
 \begin{equation}\label{7.30}
  \vn_k(x) :=\left( \begin{array}{c} \cos(2k\pi x) \\ \sin(2k\pi x) \end{array} \right) 
 \end{equation}
 satisfies \eqref{7.29} for every $k\in \mathbb{Z}$. In order to show that these are strong local minimizers, we just need to show that we satisfy all of the conditions which are required in Theorem \ref{Strong sufficient Taheri}.
 
 \begin{proposition} \label{Prop6.1}
  
  For each $k \in \mathbb{Z}$ the function $\vn_k$ is a strict strong local minimizer.
  
 \end{proposition}
 \begin{proof}
 
 We begin by noting that we could apply either Theorem \ref{Strong sufficiency} or Theorem \ref{Strong sufficient Taheri} to this problem as our domain $\Omega$ is of class $C^1$. We will use Theorem \ref{Strong sufficient Taheri} because it is a little simpler and has a slightly stronger conclusion. Our Lagrangian is smooth and convex with respect to the gradient so we know that \eqref{6.3} is satisfied. Hence in order to apply Theorem \ref{Strong sufficient Taheri} we just need to show the positivity of the second variation. We take some $\vv \in C_0^\infty \left( (0,1),\mathbb{R}^2\right)$ and for ease of notation, during this proof we will denote $\vw'_0$ simply by $\vw'$ and similarly $\vw''_0$ by $\vw''$. We remind the reader that they have the explicit form as given in \eqref{1.11}.
  
  \vspace{3 mm}
  
  We fix an index $k$. Since $\vw'' \in W^{1,1}_0(0,1)$ and $\vn_k \in W^{1,\infty}(0,1)$, they are an admissible pair of functions in \eqref{6.3} using a basic density argument. Using this fact, together with the easily verifiable identity $|\vw'|^2=-\vn_k\cdot \vw''$, the second variation simplifies as follows
  \begin{equation}\label{7.31}
  \begin{split}
   \left.\frac{ d^2}{d t^2} I \left( \frac{\vn_k+t \vv}{|\vn_k + t \vv|} \right)  \right|_{t=0} &= 2\int_0^1 |\nabla\vw'|^2+\nabla \vw'':\nabla \vn_k\,dx \\&= 2\int_0^1 |\nabla \vw'|^2+|\nabla \vn_k|^2 \left( \vw'' \cdot \vn_k \right) \,dx \\
   &=2\int_0^1 |\nabla \vw'|^2-|\nabla \vn_k|^2 |\vw'|^2 \,dx \\
   & =2\int_0^1 |\nabla \vw'|^2-4k^2\pi^2 |\vw'|^2\,dx.
  \end{split}
  \end{equation}
  To show this is positive we note that $\vw'\cdot \vn_k=0$, therefore we can denote $\vw'$ by
  \begin{equation}\label{7.32}
   \vw':=f(x)\left( \begin{array}{c} \sin(2k\pi x) \\ -\cos(2k\pi x ) \end{array} \right),
  \end{equation}
  for some $f \in W^{1,2}_0\left( 0,1 \right)$. Substituting this into \eqref{7.31} simplifies the expression to
  \begin{equation}\label{7.33}
  \begin{split}
   \left.\frac{ d^2}{d t^2} I \left( \frac{\vn_k+t \vv}{|\vn_k + t \vv|} \right)  \right|_{t=0} &=\int_0^1 f'(x)^2\,dx \\
   &\geqslant \pi^2 \int_0^1 f(x)^2\,dx \\
   & = \pi^2 \int_0^1 |\vw'|^2 \,dx.
  \end{split}
  \end{equation}
  
  The inequality above is simply a Poincar\'{e} inequality with the optimal constant for $W^{1,2}_0(0,1)$ \cite{dym1972fourier}. Now we know that \eqref{7.33} holds we proceed to prove that for some $\gamma>0$
  \begin{equation}\label{7.34}
   \left.\frac{ d^2}{dt^2} I \left( \frac{\vn_k+t\vv}{|\vn_k + t\vv|} \right)  \right|_{t=0} \geqslant \gamma|| \vw'||_{1,2}^2,
  \end{equation}
  for all test functions $\vv$. For a contradiction we suppose \eqref{7.34} does not hold, then we can argue in a similar fashion to the previous proof. For $j=1,2\dots$ there exist $\vv_j \in W^{1,2}_0(0,1)$, such that 
  \begin{equation}\label{7.35}
   ||\vw_j'||_{1,2}=1 \quad \forall \,j \quad \text{and} \quad \left.\frac{ d^2}{d t^2} I \left( \frac{\vn_k+t \vv_j}{|\vn_k + t \vv_j|} \right)  \right|_{t=0} \rightarrow 0 \quad \text{as} \quad j\rightarrow \infty.
  \end{equation}
  The estimate \eqref{7.33} tells us that 
  \begin{equation}\label{7.36}
   \int_0^1 |\vw_j'|^2 \,dx \rightarrow 0.
  \end{equation}
 However since the second variation converges to zero it is clear from \eqref{7.31} that we must also have 
 \begin{equation}\label{7.37}
  \int_0^1 |\nabla \vw_j'|^2\,dx \rightarrow 0. 
 \end{equation}
 Equations \eqref{7.36} and \eqref{7.37} clearly contradict the fact that $||\vw_j'||_{1,2}=1$ for all $j$. So the second variation is strictly positive around each $\vn_k$. Thus Theorem \ref{Strong sufficient Taheri} gives us the assertion.
 \qed
 \end{proof}

 \vspace{3 mm}
 
 This result can be intuitively grasped from the perspective of the topology of $\mathbb{S}^1$. This is the only sphere which is not simply connected so that each of these maps $\vn_k$ are not homotopic to each other. Thus there is no way of moving between these states without traversing a potential well of infinte energy. However similar results exist even if the domain is simply connected; Brezis \& Coron \cite{brezis1983large} and Jost \cite{jost1984dirichlet} both proved related results for maps into $\mathbb{S}^2$. These results all reinforce the idea that global minimizers for constrained problems is a difficult issue. Proving a global minimizer sufficiency result may not be possible for a general constrained problem without additional topological constraints on the manifold. We were able to circumvent this issue when examining local behaviour because the topology of the manifold is negated at an $L^\infty$ local level.

\vspace{3 mm}

In terms of the methods we have used, relating the constrained functional to an unconstrained one via a projection cannot be used straightforwardly to study global properties, or indeed $L^p$ local minimizers for $1\leqslant p<\infty$. This is because if we take an arbitrary $\vm\in\mathcal{B}$ we cannot perform our analysis unless $\vm$ is $L^\infty$ close to the manifold so that $P(\vm)$ is uniquely defined. One final remark about this paper is that although we considered the case of closed, bounded $C^4$ manifolds in $\mathbb{R}^k$, one could potentially go further. The condition we really required was a locally unique projection onto the manifold which was itself $C^3$. Therefore the analysis should all hold if the constraint is $\vn(x) \in M$ almost everywhere, and $M$ satisfies this condition.

\section*{Acknowledgements}
I would like to thank my supervisor Professor John Ball for his many helpful discussions and ideas which helped formulate this work. I would also like to thank Chris Newton and HP-Labs for introducing me to the study of cholesteric liquid crystals.

\bibliographystyle{plain}
\bibliography{refs}


\end{document}